\documentclass{article}

\usepackage[nonatbib,preprint]{neurips_2020}

\usepackage[utf8]{inputenc} % allow utf-8 input
\usepackage[T1]{fontenc}    % use 8-bit T1 fonts
\usepackage{hyperref}       % hyperlinks
\usepackage{url}            % simple URL typesetting
\usepackage{booktabs}       % professional-quality tables
\usepackage{amsfonts}       % blackboard math symbols
\usepackage{nicefrac}       % compact symbols for 1/2, etc.
\usepackage{microtype}      % microtypography

\usepackage{array}
\usepackage{tabularx}

\usepackage{makecell}
\usepackage{xcolor}
\hypersetup{
     colorlinks = true,
     linkcolor = blue,
     anchorcolor = blue,
     citecolor = blue,
     filecolor = blue,
     urlcolor = blue
     }

\usepackage{amsmath}
\usepackage{amsfonts}
\usepackage{amssymb} 
\usepackage{threeparttable}

\usepackage{algorithm}
\usepackage{algorithmic}

\usepackage{graphicx}
\usepackage{subcaption}

\def\E{\mathbb{E}}
\def\R{\mathbb{R}}
\def\O{\mathcal O}

\def\V{{\mathcal V}}
\def\F{{\mathcal F}}

\def\I{{\mathcal I}}

\def\1{{\bf 1}}

\def\e{\boldsymbol \epsilon}

\def\bbx{\bar{\mathbf x}}
\def\hbg{\hat{\mathbf g}}
\def\bbg{\bar{\mathbf g}}
\def\tbg{\tilde{\mathbf g}}

\def\bx{{\mathbf x}}

\def\bz{{\mathbf z}}
\def\bg{{\mathbf g}}
\def\bu{{\mathbf u}}

\def\by{{\mathbf y}}

\usepackage{amsthm}
\newtheorem{theorem}{Theorem}
\newtheorem{lemma}{Lemma}
\newtheorem*{lemmapp}{Lemma}
\newtheorem{corollary}{Corollary}
\newtheorem{assumption}{Assumption}
\title{Local SGD With a Communication Overhead Depending Only on the Number of Workers}

\author{%
  Artin Spiridonoff\\
  Division of Systems Engineering\\
  Boston University\\
  Boston, MA 02215 \\
  \texttt{artin@bu.edu} \\
  \And
  Alex Olshevsky \\
  Division of Systems Engineering \\
  Boston University\\
  Boston, MA 02215 \\
  \texttt{alexols@bu.edu} \\
  \AND
  Ioannis Ch. Paschalidis \\
  Division of Systems Engineering \\
  Boston University\\
  Boston, MA 02215 \\
  \texttt{yannisp@bu.edu} \\
}

\begin{document}

\maketitle

\begin{abstract}
	We consider speeding up stochastic gradient descent (SGD) by parallelizing it across multiple workers. We assume the same data set is shared among $n$ workers, who can take SGD steps and coordinate with a central server. Unfortunately, this could require a lot of communication between the workers and the server, which can dramatically reduce the gains from parallelism. The Local SGD method, proposed and analyzed in the earlier literature, suggests machines should make many local steps between such communications. While the initial analysis of Local SGD showed it needs $\Omega ( \sqrt{T} )$ communications for $T$ local gradient steps in order for the error to scale proportionately to $1/(nT)$, this has been successively improved in a string of papers, with the state-of-the-art requiring  $\Omega \left( n \left( \mbox{ polynomial in  log } (T) \right) \right)$ communications. In this paper, we give a new analysis of Local SGD. A consequence of our analysis is that Local SGD can achieve an error that scales as $1/(nT)$ with only a fixed number of communications independent of $T$: specifically, only $\Omega(n)$  communications are required.
\end{abstract}

\section{Introduction}
%Parallel SGD \\
Stochastic Gradient Descent (SGD) is a widely used algorithm to minimize a convex or non-convex function $F$ in which model parameters are updated iteratively as follows:
\begin{align*}
	\bx^{t+1} = \bx^t - \eta_t \hbg^t,
\end{align*}
where $\hbg^t$ is a stochastic gradient of $F$ at $\bx^t$ and $\eta_t$ is the learning rate.
This algorithm can be naively parallelized by adding more workers independently  to compute a gradient and then average them at each step to reduce the variance in estimation of the true gradient $\nabla F(\bx^t)$ \cite{dekel2012optimal}. This method requires each worker to share their computed gradients with each other at every iteration. 

%communication bottle neck \\
However, it is widely acknowledged that communication is a major bottleneck of this method for large scale optimization applications \cite{mcmahan2016communication,konevcny2016federated,lin2017deep}.
Often, mini-batch parallel SGD is suggested to address this issue by increasing the computation to communication ratio.
Nonetheless, too large mini-batch size might degrades the performance \cite{lin2018don}. Along the same lines of increasing compute to communication, \textit{local} SGD has been proposed to reduce communications \cite{mcmahan2016communication,dieuleveut2019communication}. In  this method, workers compute (stochastic) gradients and update their parameters locally, and communicate only once in a while to obtain the average of their parameters.
Local SGD improves the communication efficiency not only by reducing the number of communication rounds, but also alleviates the synchronization delay caused by waiting for slow workers and evens out the variations in workers' computing time \cite{wang2018cooperative}.

On the other hand, since individual gradients of each worker are calculated at different points, this method introduces residual error as opposed to fully synchronous SGD. Therefore, there is a trade-off between having fewer communication rounds and introducing additional errors to the gradient estimates.

%related work \\
The idea of making local updates is not new and has been used in practice for a while  \cite{konevcny2016federated}. However, until recently, there have been few successful efforts to analyze Local SGD theoretically and therefore it is not fully understood yet.
The paper \cite{zhang2016parallel} shows that for quadratic functions, when the variance of the noise is higher far from the optimum, frequent averaging leads to faster convergence.
One of the main questions we want to ask is: how many communication rounds are needed for Local SGD to have the \textit{same} convergence rate of a synchronized parallel SGD while achieving performance that linearly improves in the number of workers?

% There has been a growing recent interest in answering the question above by theoretical guarantees.
\cite{stich2018local} was among the earlier works that tried to answer this question for general strongly convex and smooth functions and showed that the communication rounds can be reduced up to a factor of $H = \O(\sqrt{T/n})$, without affecting the asymptotic convergence rate (up to constant factors), where $T$ is the total number of iterations and $n$ is number of parallel workers. 
%and $b$ is the mini-batch size.

Focusing on smooth and possibly non-convex functions which satisfy  a Polyak-Lojasiewicz condition, \cite{haddadpour2019local} demonstrates that only $R = \Omega((Tn)^{1/3})$ communication rounds are sufficient to achieve asymptotic performance that scales proportionately to $1/n$.

More recently, \cite{khaled2019tighter} and \cite{stich2019error} improve upon the previous works by showing linear-speed up for Local SGD with only $\Omega \left(n \mbox{ poly log }(T) \right)$ communication rounds when data is identically distributed among workers and $F$ is strongly convex. Their works also consider the cases when $F$ is not necessarily strongly-convex as well as the case of data being heterogeneously distributed among workers in \cite{khaled2019tighter}.

\begin{table}
    \caption{Comparison of Similar Works}
    \label{table: comparison}
    \begin{threeparttable}[b]
    \centering
    \def\arraystretch{2}
    \resizebox{\textwidth}{!}{
    \begin{tabular}{| c c c c c |}
    %\toprule
    \hline
    % \multicolumn{2}{c}{Part} \\
    % \cmidrule(r){1-2}
    \thead{noise\\model} & \thead{$H=T$ \tnote{a}\\convergent} & \thead{communication\\rounds, $R$} & \thead{convergence rate,\\$F(\hat \bx^t) - F^*$} & Reference \\
    % \cmidrule(r){3-4} \\
    % \midrule
    \hline
    uniform  & no  & $\Omega(\sqrt{Tn})$ & $\O(\frac{\xi^0}{R^3} + \frac{\sigma^2}{\mu n T} + \frac{\kappa G^2}{\mu R^2})$
    \tnote{b}& \cite{stich2018local} \\
    \thead{uniform with\\strong-growth \tnote{c}} & no & $\Omega((Tn)^{1/3})$ &
    $\O(\frac{\xi^0}{R^3}+\frac{\kappa \sigma^2}{\mu n T} + \frac{\kappa^2 \sigma^2}{\mu n T R})$ & \cite{haddadpour2019local} \\
    \thead{uniform with\\strong-growth} & no & $ \Omega(n*\text{poly-log} (T))$ &
    $\tilde \O(\frac{\kappa n H \xi^0}{\exp(R/(\kappa n))} + \frac{\sigma^2}{\mu n T})$ \tnote{d} & \cite{stich2019error} \\
    uniform & yes & $ \Omega(n*\text{poly-log} (T))$ &
    $\tilde \O( \frac{\kappa \xi^0}{T^2} + \frac{\kappa \sigma^2}{n \mu T} + \frac{\kappa^2 \sigma^2}{\mu T R} )$ & \cite{khaled2019tighter} \\
    \thead{uniform with\\strong-growth} & yes & $\Omega(n)$ &
    $\O(\frac{\kappa^4 \ln(TR^{-2})\xi^0}{T^2} + \frac{\kappa \sigma^2}{\mu n T} + \frac{\kappa^2 \sigma^2}{\mu T R})$ & \textbf{This Paper} \\
    % \bottomrule
    \hline
    \end{tabular}
    }
    \begin{tablenotes}
    \item [a] $H$ is the length of inter-communication intervals.
    \item [b] $G$ is the uniform upper bound assumed for the $l_2$ norm of gradients in the corresponding work.
    \item [c] This noise model is defined in Assumption \ref{as: g}.
    \item [d] $\tilde \O (.)$ ignores the poly-logarithmic and constant factors. 
    \end{tablenotes}
    \end{threeparttable}
\end{table}

In this work, we focus on smooth and strongly-convex functions with a very general noise model. The main contribution of this paper is to propose a communication strategy which requires only $R = \Omega(n)$ communication rounds to achieve performance that scales as $1/n$ in the number of workers. To the best of the authors' knowledge, this is the only work to show this result (without additional poly-logarithmic terms and constants).
Our analysis can also recover some of the best known rates for special cases, e.g., when $H$ is constant, where $H$ is defined as the length of intercommunication intervals.
A summary of our results compared to the available literature can be found in Table \ref{table: comparison}.

% We also recover the best known bounds for when the communication intervals are fixed.
% In addition, our bounds are for the last iterate as opposed to a weighted average. We also don't assume bounded gradients.

The rest of this paper is organized as follows. In the following subsection we outline the related literature and ongoing works. In Section  \ref{sec: Problem} we define the main problem and state our assumptions. We present our theoretical findings in Section \ref{sec: convergence} and the sketch of proofs in Section \ref{sec: sketch of proof}, followed by numerical experiments in Section \ref{sec: Numerical } and conclusion remarks in Section \ref{sec: conclustion}.

\subsection{Related Works}
There has been a lot of effort in the recent research to take into account the communication delays and training time in designing faster algorithms \cite{mcdonald2010distributed,zhang2015deep,bijral2016data,kairouz2019advances}.
See \cite{tang2020communication} for a comprehensive survey of communication efficient distributed training algorithms considering both system-level and algorithm-level optimizations.

% Lower Bounds
 Many works study the communication complexity of distributed methods for convex optimization \cite{arjevani2015communication} \cite{woodworth2020local}  and statistical estimation \cite{zhang2013information}.
\cite{woodworth2020local} presents a rigorous comparison of Local SGD with $H$ local steps and mini-batch SGD with  $H$ times larger mini-batch size and the same number of communication rounds (we will refer to such a method as large mini-batch SGD) and show regimes in which each algorithm performs better: they show that Local SGD is strictly better than large mini-batch SGD when the functions are quadratic. Moreover, they prove a lower bound on the worst case of Local SGD that is higher than the worst-case error of large mini-batch SGD in a certain regime.
\cite{zhang2013information} studies the minimum amount of communication required to achieve centralized minimax-optimal rates by establishing lower bounds on minimax risks for distributed statistical estimation under a communication budget.

% nonconvex
A parallel line of work studies the convergence of Local SGD with non-convex functions \cite{zhou2017convergence}.
\cite{yu2019parallel} was among the first works to present provable guarantees of Local SGD with linear speed up.
\cite{wang2018cooperative} and \cite{koloskova2020unified} present unified frameworks for analyzing decentralized SGD with local updates, elastic averaging or changing topology.
The follow-up work \cite{wang2018adaptive} presents ADACOMM, an adaptive communication strategy that starts with infrequent averaging and then increases the communication frequency in order to achieve a low error floor. They analyze the error-runtime trade-off of Local SGD 
% with constant learning rate over smooth
with nonconvex functions and propose communication times to achieve faster runtime.

% one-shot averaging
In \textit{One-Shot Averaging} (OSA), workers perform local updates with no communication during the optimization until the end when they average their parameters. This method can be seen as an extreme case of Local SGD with $H=T$, on the opposite end of synchronous SGD \cite{mcdonald2009efficient,zinkevich2010parallelized,zhang2013communication,rosenblatt2016optimality,godichon2017rates}. \cite{dieuleveut2019communication} provides non-asymptotic analysis of mini-batch SGD and one-shot averaging as well as regimes in which mini-batch SGD could outperform one-shot averaging.

% Compression
Another line of work reduces the communication by compressing the gradients and hence limiting the number of bits transmitted in every message between workers \cite{lin2017deep,alistarh2017qsgd,wangni2018gradient,stich2018sparsified,stich2019error}.

% Asynchronous Parallel/Decentralized methods
Asynchronous methods have been studied widely due to their advantages over synchronous methods which suffer from synchronization delays due to the slower workers \cite{olshevsky2018robust}.
\cite{wang2019matcha} studies the error-runtime trade-off in decentralized optimization and proposes MATCHA, an algorithm which parallelizes inter-node communication by decomposing the topology into matchings. 
\cite{hendrikx2019accelerated} provides an accelerated stochastic algorithm for decentralized optimization of finite-sum objective functions that by carefully balancing the ratio between communications and computations match the rates of the best known sequential algorithms while having the network scaling of optimal batch algorithms.
However, these methods are relatively more involved and they  often require full knowledge of the network, solving a semi-definite program and/or calculating communication probabilities (schedules).

\subsection{Notation}
For a positive integer $s$, we define $[s]:=\{1,\ldots,s\}$. We use bold letters to represent vectors. We denote vectors of all $0$s and $1$s by $\mathbf{0}$ and $\mathbf{1}$, respectively. We use $\Vert \cdot \Vert$ for the Euclidean  norm.

\section{Problem Formulation}\label{sec: Problem}
Suppose there are $n$ workers  $\V=\{1, \ldots, n\}$, trying to minimize $F:\R^d \rightarrow \R$ in parallel.
We assume all workers have access to $F$ through noisy gradients. In Local SGD, workers perform local gradient steps and occasionally calculate the average of all workers' iterates. 

Having access to the same objective function $F$ is of special interest if the data is stored in one place accessible to all machines or is distributed identically among workers with no memory constraints. We hope that results presented here can be extended to applications with heterogeneous data distributions \cite{khaled2019tighter}.

We will make the following additional assumptions.

\begin{assumption}\label{as: F}
	Function $F:\R^d\rightarrow \R$ is differentiable, $\mu$-strongly convex and $L$-smooth for $L\geq\mu>0$. In particular, 
	\begin{align*}
		\frac{\mu}{2}\Vert \bx - \by \Vert^2 \leq F(\by) - F(\bx) - \langle \nabla F(\bx), \by - \bx \rangle \leq \frac{L}{2}\Vert \bx - \by \Vert^2, \qquad \forall x,y\in \R^d.
	\end{align*}
	We define $\kappa=L/\mu$ to be the condition number of $F$.
\end{assumption}
We make the following assumption on the noise of the stochastic gradients.
\begin{assumption}\label{as: g}
	Each worker $i$ has access to a gradient oracle which returns an unbiased estimate of the true gradient in the form $\hbg_i(\bx) = \nabla F(\bx) + \e_i$, such that $\e_i$ is a zero-mean conditionally independent random noise with its expected squared norm error bounded as
	\begin{align*}
	\E[\e_i] = \mathbf 0, \qquad \E[\Vert \e_i \Vert^2|\bx] \leq c\Vert \nabla F(\bx) \Vert^2 + \sigma^2,
	\end{align*}
	where $\sigma^2,c\geq0$ are constants.
\end{assumption}
To save space, we define $\hbg_i^t := \hbg(\bx_i^t)$ as the stochastic gradient of node $i$ at iteration $t$, and $\bg_i^t = \nabla F(\bx_i^t)$ as the true gradient at the same point.

The noise model of Assumption \ref{as: g} is very general and it includes the common case with uniformly bounded squared norm error when $c=0$. As it is noted by \cite{zhang2016parallel}, the advantage of periodic averaging compared to one-shot averaging only appears when $c/\sigma^2$ is large. Therefore, to study Local SGD, it is important to consider a noise model as in Assumption \ref{as: g} to capture the effects of frequent averaging.
Among the related works mentioned in Table \ref{table: comparison}, only \cite{stich2019error} and \cite{haddadpour2019local} analyze this noise model while the rest study the special case with $c=0$.
SGD under this noise model with $c>0$ and $\sigma^2=0$ was first studied in \cite{schmidt2013fast} under the name \textit{strong-growth condition}. Therefore we refer to the noise model considered in this work as \textit{uniform with strong-growth}.

In Local SGD, each worker $i$ holds a local parameter $\bx_i^t$ at iteration $t$ and a set $\I \subset [T]$ of communication times, and performs the following update:
\begin{align}\label{eq: Local SGD process}
\bx_i^{t+1} = \begin{cases}
x_i^t - \eta_t \hbg_i^t, \qquad &\text{if } t+1 \notin \I,\\
\frac{1}{n}\sum_{j=1}^n (\bx_j^t - \eta_t \hbg_j^t),
\qquad &\text{if } t+1 \in \I.
\end{cases}
\end{align}
When $\I = [T]$, we recover the fully synchronized parallel SGD, while $\I = \{T\}$ recovers one-shot averaging.
The pseudo code for Local SGD is provided as Algorithm \ref{alg: Local SGD}.
\begin{algorithm}
	\caption{Local SGD}
	\begin{algorithmic}[1]\label{alg: Local SGD}
		\STATE Input $\bx_i^0 = \bx^0$ for $i \in [n]$, total number of iterations $T$, the step-size sequence $\{\eta_t\}_{t=0}^{T-1}$ and $\I \subseteq [T]$
		\FOR{$t=0,\ldots,T-1$}
		\FOR{$j=1,\ldots,n$}
		\STATE evaluate a stochastic gradient $\hbg_j^t$
		\IF{$t+1 \in \mathcal{I}$}
		\STATE $\bx_j^{t+1} = \frac{1}{n}\sum_{i=1}^n (\bx_i^t - \eta_t \hbg_i^t)$
		\ELSE
		\STATE $\bx_j^{t+1} = \bx_j^t - \eta_t \hbg_j^t$
		\ENDIF
		\ENDFOR
		\ENDFOR
	\end{algorithmic}
\end{algorithm}

The main goal of this paper is to study the effect of communication times on the convergence of the Local SGD and provide better theoretical guarantees. In what follows, we claim that by carefully choosing the step size, linear speed-up of parallel SGD can be attained with only a small number of communication instances.

\section{Convergence Results}\label{sec: convergence}
In this section we present our convergence results for Local SGD.
In the following theorem, we show an upper bound for the sub-optimality error, in the sense of function value, for any choice of communication times $\I$.

Before proceeding with our results, let us introduce some notation. Let $0=\tau_0<\tau_1 < \ldots < \tau_R = T$ be the communication times. Define $H_i := \tau_{i+1} - \tau_i$, as the length of $i+1$-th inter-communication interval, for $i=0,\ldots,k-1$. Moreover, define $\bbx^t := (\sum_{i=1}^n \bx_i^t)/n$ as the the average of the iterates of all workers. Notice that $\bx_i^t = \bbx^t$ for $t\in \I$.

The main results of this paper will be obtained by specializing the following bound. 

\begin{theorem}\label{thm: general}
	Suppose Assumptions \ref{as: F} and \ref{as: g} hold. Choose $\beta \geq 2\kappa^2 $ and communication times $\I = \{\tau_i|i=1,\ldots,R\}$ such that it holds
	\begin{align}\label{eq: beta condition}
	    9\kappa^2c \ln(1 + \frac{H_i - 1}{\tau_i + \beta}) + 2\kappa(1+\frac{c}{n}) - (\tau_i +1 + \beta)\leq 0, \qquad i=0,\ldots,R-1.
	\end{align}
	Set $\eta_k = 2/ (\mu (k+\beta))$. Then, using Algorithm \ref{alg: Local SGD}, we have 
	\begin{align}\label{eq: opt E3}
	% \E[F(\bbx^T)] - F^* \leq (\E[F(\bbx^0)] - F^*) \left(\frac{\beta-1}{T+\beta - 1}\right)^2 +  \frac{2L \sigma^2}{n \mu^2 T} + \frac{9L^2(n-1) \sigma^2}{n \mu^3 (T+\beta - 1)^2} \sum_{t=0}^{T-1} \frac{t-\tau(t)}{t+\beta},
	\E[F(\bbx^T)] - F^* \leq \frac{\beta^2 (F(\bbx^0) - F^*)}{T^2}  +  \frac{2L \sigma^2}{n \mu^2 T} + \frac{9L^2 \sigma^2}{ \mu^3 T^2} \sum_{t=0}^{T-1} \frac{t-\tau(t)}{t+\beta},
	\end{align}
	where $F^* := \min_{\bx} F(\bx)$ and $\tau(t) := \max\{t' \in \mathcal{I}| t' \leq t\}$ is the most recent communication time.
\end{theorem}
The last term in Equation \eqref{eq: opt E3} is due the to disagreement between workers (consensus error), introduced by local computations without any communication. As the inter-communication intervals become larger, $t-\tau(t)$ becomes larger as well and increases the overall optimization error. This term explains the trade-off between communication efficiency and the optimization error.

Note that condition \eqref{eq: beta condition} is mild. For instance, it suffices to set $\beta \geq \max \{9\kappa^2 c \ln(1 + T/(2\kappa^2)) + 2\kappa(1+c/n), 2\kappa^2\}$. Moreover, the bound in \eqref{eq: opt E3} is for the last iterate $T$, and does not require keeping track of a weighted average of all the iterates.

Theorem \ref{thm: general} not only bounds the optimization error, but introduces a methodological approach to select the communication times to achieve smaller errors. For the scenarios when the user can afford to have a certain number of a communications, they can select $\tau_i$ to minimize the last term in \eqref{eq: opt E3}.

We next discuss the implications of Theorem \ref{thm: general} under various conditions. 

\paragraph{One-Shot Averaging.} Plugging $H=T$ in Theorem \ref{thm: general}, we obtain a convergence rate of $\O(\kappa^2 \sigma^2/(\mu T))$ without any linear speed-up. Among previous works, only \cite{khaled2019tighter} shows a similar result.

\subsection{Fixed-Length Intervals}
A simple way to select the communication times $\I$, is to split the whole training time $T$ to $R$ intervals of length at most $H$. Then we can use the following bound in Equation \eqref{eq: opt E3}, 
\begin{align*}
\sum_{t=0}^{T-1} \frac{t-\tau(t)}{t+\beta} \leq (H-1)\sum_{t=0}^{T-1} \frac{1}{t+\beta} \leq (H-1) \ln(1+\frac{T}{ \beta - 1}).
\end{align*}
We state this result formally in the following corollary.
\begin{corollary}
	Suppose assumptions of Theorem \ref{thm: general} hold and in addition, workers communicate at least once every $H$ iterations. Then,
	\begin{align}\label{eq: fixed-int}
	\E[F(\bbx^T)] - F^* \leq \frac{\beta^2(F(\bbx^0) - F^*)}{T^2} +  \frac{2L \sigma^2}{n \mu^2 T} + \frac{9L^2 \sigma^2 (H-1)}{ \mu^3 T^2} \ln(1 + \frac{T}{\beta - 1}).
	\end{align}
\end{corollary}

\paragraph{Linear Speed-Up.}
Setting $H= \O(T/(n \ln(T)))$ we achieve linear-speed up in the number of workers, which is equivalent to a communication complexity of $R = \Omega(n \ln(T))$. To the best of the authors' knowledge, this is the tightest communication complexity that is shown to achieve linear speed-up. \cite{khaled2019tighter} and \cite{stich2019error} have shown a similar communication complexity, however with slightly higher degrees of dependence on $\ln(T)$, e.g., $R=\Omega(n \ln(T)^2)$ in \cite{khaled2019tighter}.

\paragraph{Recovering Synchronized SGD.} When $H=1$, the the last term in \eqref{eq: fixed-int} disappears and we recover the convergence rate of parallel SGD, albeit, with a worse dependence on $\kappa$.

\subsection{Varying Intervals}
In the previous subsection, we observed that with our current analysis, having fixed-length inter-communication intervals, linear speed-up can be achieved with only $\Omega(n\ln(T))$ rounds of communications. A natural question that might arise is whether we can improve the result above even further. 

Let us allow consecutive inter-communication intervals, i.e., $H_i := \tau_{i+1} - \tau_i$, grow linearly, where $0=\tau_0 <\tau_1<\ldots<\tau_R = T$ are the communication times. The following Theorem presents a performance guarantee for this choice of communication times.

% \begin{theorem}\label{thm: logT}
% 	Suppose assumptions of Theorem \ref{thm: general} hold. Set $k = \lceil \ln(T) \rceil$ and $ a = \lceil 2T/k^2 \rceil$. Then set $H_i = a(i+1)$ and $\tau_{i+1} = \min(\tau_i + H_i, T)$ for $i=0,\ldots,k-1$. We obtain,
% 	\begin{align}\label{eq: opt log}
% 	\E[F(\bbx^T)] - F^* \leq \frac{ \beta^2 (F(\bbx^0) - F^*)}{T^2} +  \frac{2L \sigma^2}{n \mu^2 T} 
% 	+ \frac{18L^2 \sigma^2}{\mu^3 T} \left(\frac{2}{\ln(T)} + \frac{\ln(T)+1}{T} \right).
% 	\end{align}
% \end{theorem}

\begin{theorem}\label{thm: logT}
	Suppose Assumptions \ref{as: F} and \ref{as: g} hold. Choose the maximum number of communications $1 \leq R \leq \sqrt{2T}$ and set $a:=\lceil 2T/R^2 \rceil \geq 1$, $H_i = a(i+1)$ and $\tau_{i+1} = \min(\tau_i + H_i, T)$ for $i=0,\ldots,R-1$. Choose $\beta \geq \max \{2 \kappa^2, 9\kappa^2c \max\{\ln(3), \ln(1+T/(R^2\kappa^2))\} + 2\kappa(1+c/n) \}$ and set $\eta_t = 2/ \mu (t+\beta)$. Then using Algorithm \ref{alg: Local SGD} we have,
	\begin{align}\label{eq: opt general linear H}
	\E[F(\bbx^T)] - F^* \leq \frac{ \beta^2 (F(\bbx^0) - F^*)}{T^2} +  \frac{2L \sigma^2}{n \mu^2 T} 
	+ \frac{72L^2 \sigma^2}{\mu^3 TR}.
	\end{align}
\end{theorem}

The choice of communication times in Theorem \ref{thm: logT} aligns with the intuition that workers need to communicate more frequently at the beginning of the optimization. As the the step-sizes become smaller and workers' local parameters get closer to the global minimum, they diverge more slowly from each other and, hence, less communication is required to re-align them.

\paragraph{Linear Speed-Up.}
Choosing communication rounds $R = \Omega(n)$, we achieve an error that scales as $1/(nT)$ in the number of workers when $T=\Omega(n^2)$. {\bfseries This is the main result of this paper}: it shows that we can get a linear speedup in the number of workers by simply increasing the number of iterations while keeping the total number of communications bounded. 

\section{Sketch of Proof}\label{sec: sketch of proof}
Here we give an outline of the proofs for the results presented in this paper. The proof of the following lemmas are left to the Appendix.

\paragraph{Perturbed Iterates.}
A common approach in analyzing parallel algorithms such as Local SGD is to study the evolution of the sequence $\{\bbx^t\}_{t\geq0}$. We have, 
\begin{align}\label{eq: average x update}
    \bbx^{t+1} = \bbx^t - \frac{\eta_t}{n}\sum_{i=1}^n \hbg_i^t 
    = \bbx^t - \eta_t \tbg^t,
\end{align}
where $\tbg^t := (\sum_{i=1}^n \hbg_i^t)/n$ is the average of the stochastic gradient estimates of all workers. 

Let us define $\xi^t:= \E[F(\bbx^t)] - F^*$ to be the optimality error. The following lemma, which is similar to a part of the proof found in \cite{haddadpour2019local}, bounds the optimality error at each iteration recursively.
\begin{lemma}\label{lem: error decay}
    Let Assumptions \ref{as: F} and \ref{as: g} hold. Then,
    \begin{align*}
	    \xi^{t+1} \leq \xi^t(1 - \mu \eta_t) + \frac{L^2 \eta_t}{2n} \E \left[ \sum_{i=1}^n \Vert \bbx^t - \bx_i^t \Vert^2 \right] + \frac{\eta_t^2 L}{2} \E[\Vert \tbg^t \Vert_2^2 ] 
	    - \frac{\eta_t}{2n} \E \left[ \sum_{i=1}^n \Vert \nabla F(\bx_i^t) \Vert^2\right].
	\end{align*}
\end{lemma}
Equipped with Lemma \ref{lem: error decay}, we can bound the consensus error ($\E[\sum_{i=1}^n \Vert \bbx^t - \bx_i^t \Vert^2]$) as well as the term $\E[\Vert \tbg^t \Vert^2]$ in the following lemmas. 

\paragraph{Consensus Error.}
In the following lemmas, we utilize the structure of the problem to bound the consensus error recursively.
\begin{lemma}\label{lem: consensus 1}
	Let Assumptions \ref{as: F} and \ref{as: g} hold. Then,
	\begin{multline}\label{eq: consensus 2}
	\E\left[ \sum_{i=1}^n \Vert \bx_i^{t+1} - \bbx^{t+1} \Vert^2 \right] \leq \E\left[ \sum_{i=1}^n \Vert \bx_i^{t} - \bbx^t \Vert^2 \right](1 - 2\eta_t \mu + \eta_t^2 L^2) \\
	+ (n-1)\eta_t^2\sigma^2
	+ (1-\frac{1}{n})\eta_t^2 c\E [\sum_{i=1}^n \Vert g_i^t \Vert^2].
	\end{multline}
\end{lemma}

This lemma, bounds how much the consensus error grows at each iteration. Of course, when workers communicate, this error resets to zero and thus, we can calculate an upper bound for the consensus error, knowing the last iteration communication occurred and the step-size sequence. The following lemma takes care of that. Before stating the following lemma, let us define $G^t := \frac{1}{n}\sum_{i=1}^n \Vert \bg_i^t \Vert^2$.

\begin{lemma}\label{lem: consensus 2}
	Let assumptions of Theorem \ref{thm: general} hold. Then,
	\begin{align}
	\E\left[ \sum_{i=1}^n \Vert \bx_i^{t} - \bbx^t \Vert^2 \right] \leq
% 	(n-1)\sigma^2\frac{9(t - \tau(t))}{\mu^2 (t+\beta)^2} + 
    9(n-1)\sum_{k=\tau(t)}^{t-1}\frac{c\E[G^k]+\sigma^2}{\mu^2(t+\beta)^2}.
	\end{align}
\end{lemma}

\paragraph{Variance.}
Our next lemma bounds $E[\Vert \tbg^t \Vert^2]$.
\begin{lemma}\label{lem: variance}
    Under Assumption \ref{as: g} we have,
    \begin{align*}
        \E [\Vert \tbg^t \Vert^2 ] \leq (1+\frac{c}{n}) %\frac{1}{n}\sum_{i=1}^n \E \left[\left\Vert \nabla F(\bx_i^t) \right\Vert^2 \right]
        \E[G^t] + \frac{\sigma^2}{n}.
    \end{align*}
\end{lemma}

The proofs of  Theorems \ref{thm: general} and \ref{thm: logT} follow from these lemmas. Due to space constraints, these proofs are given in the supplementary information.

\section{Numerical Experiments}\label{sec: Numerical }
To verify our findings and compare different communication strategies in Local SGD, we performed the following numerical experiments.

 \subsection{Quadratic Function With Strong-Growth Condition}
As discussed in \cite{zhang2016parallel,dieuleveut2019communication}, under uniformly bounded variance, one-shot averaging performs asymptotically as well as mini-batch SGD. Therefore, to fully capture the importance of the choice of communication times $\I$, we design a \textit{hard} problem, where noise variance is uniform with strong-growth condition, defined in Assumption \ref{as: g}.  Let us define $F(\bx) = \E_\zeta f(\bx, \zeta)$ where,
\begin{align}\label{eq: f-zeta}
    f(\bx, \zeta) := \sum_{i=1}^d \frac{1}{2}x_i^2 (1 + z_{1,i}) + \bx^\top \bz_2,
\end{align}
$\zeta = (\bz_1, \bz_2)$, where $\bz_1, \bz_2 \in \R^d, z_{1,i}\sim \mathcal{N}(0,c_1)$ and $z_{2,i} \sim \mathcal{N}(0,c_2)$, $\forall i \in[d]$ are random variables with normal distributions. We assume at each iteration $t$, each worker $i$ samples a $\zeta_i^t$ and uses $\nabla f(\bx, \zeta_i^t)$ as a stochastic estimate of $\nabla F(\bx)$.  It is easy to verify that $F(\bx) = (1/2)\bx^2$ is $1$-strongly convex, $F^* = 0$ and $\E_\zeta[\Vert \nabla f(\bx, \zeta) - \nabla F(\bx) \Vert^2] = c \Vert \nabla F(\bx) \Vert^2 + \sigma^2$, where $c = c_1$ and $\sigma^2 = dc_2$.

We use Local SGD to minimize $F(\bx)$ using different communication strategies. We select $c_1=9, c_2=0.25, d=3$, $n=20$ machines and $T=1000$ iterations and the step-size sequence $\eta_t = 2/\mu(t+\beta)$ with $\beta=1$. We start each simulation from the initial point of $\bx^0 = \mathbf 1_d$ and repeat each simulation $500$ times. The average of the results are reported in Figures \ref{fig: res2}(a) and \ref{fig: res2}(b). Moreover, average performance of Local SGD with different number of workers $n$ and the communication strategy proposed in this paper with $R=n$ is shown in Figure \ref{fig: res2}(c) along with the respective convergence rate of $\sigma^2/(\mu n t)$.

\begin{figure}
	\centering
	\begin{subfigure}[b]{0.325\textwidth}
		\includegraphics[width=\textwidth]{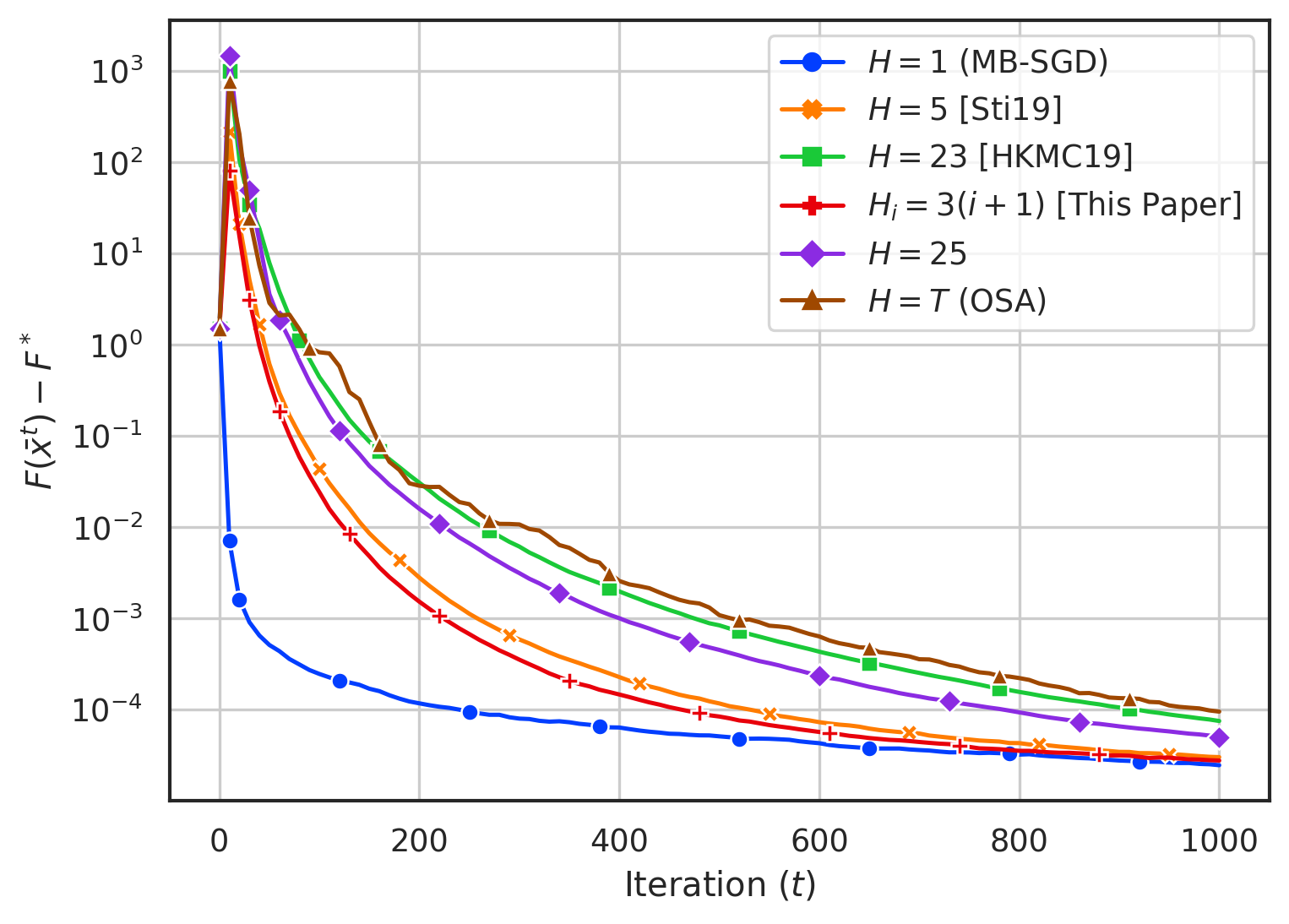}
		\caption{Error over iterations.}
	\end{subfigure}
	\begin{subfigure}[b]{0.325\textwidth}
		\includegraphics[width=\textwidth]{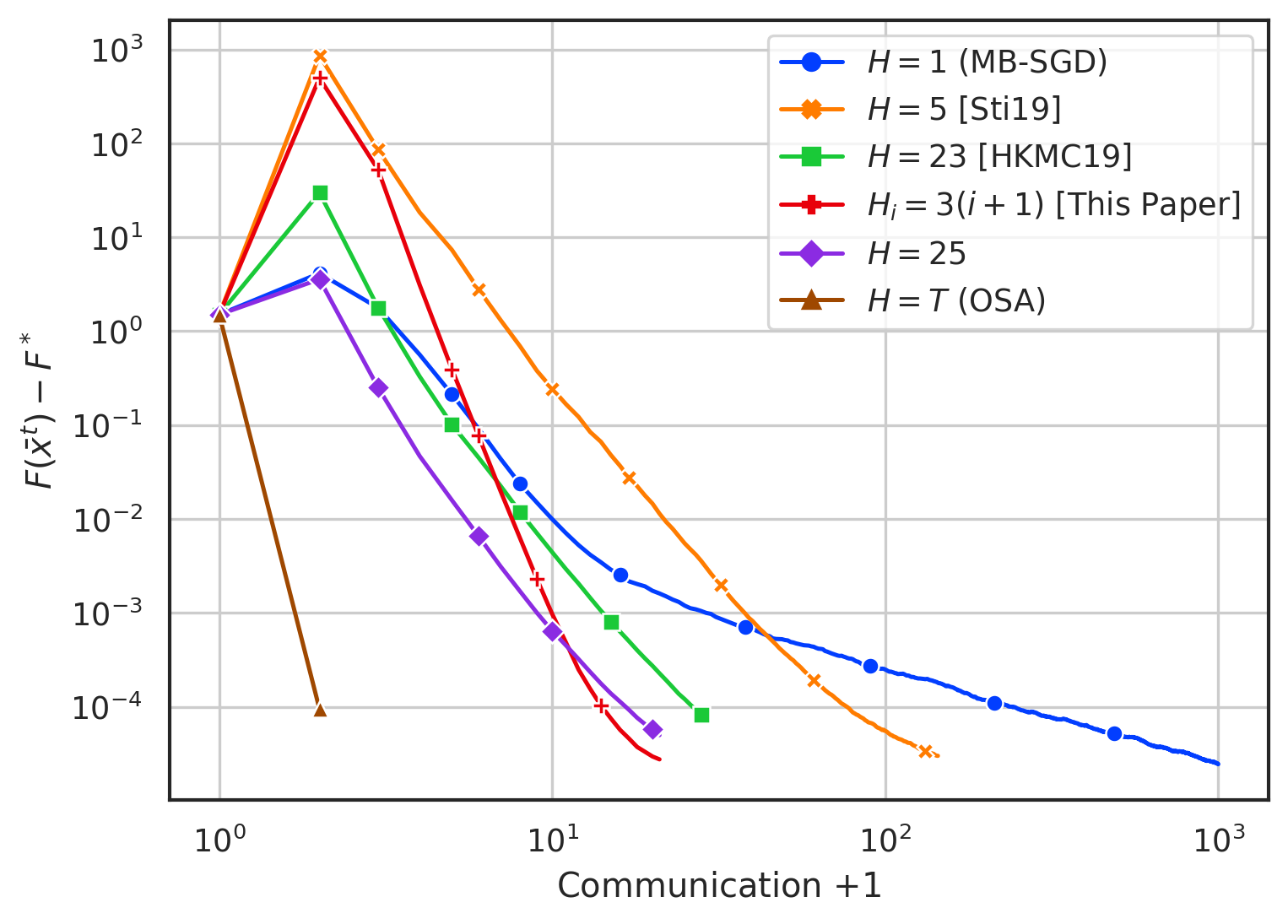}
		\caption{Error over communications.}
	\end{subfigure}
	\begin{subfigure}[b]{0.325\textwidth}
		\includegraphics[width=\textwidth]{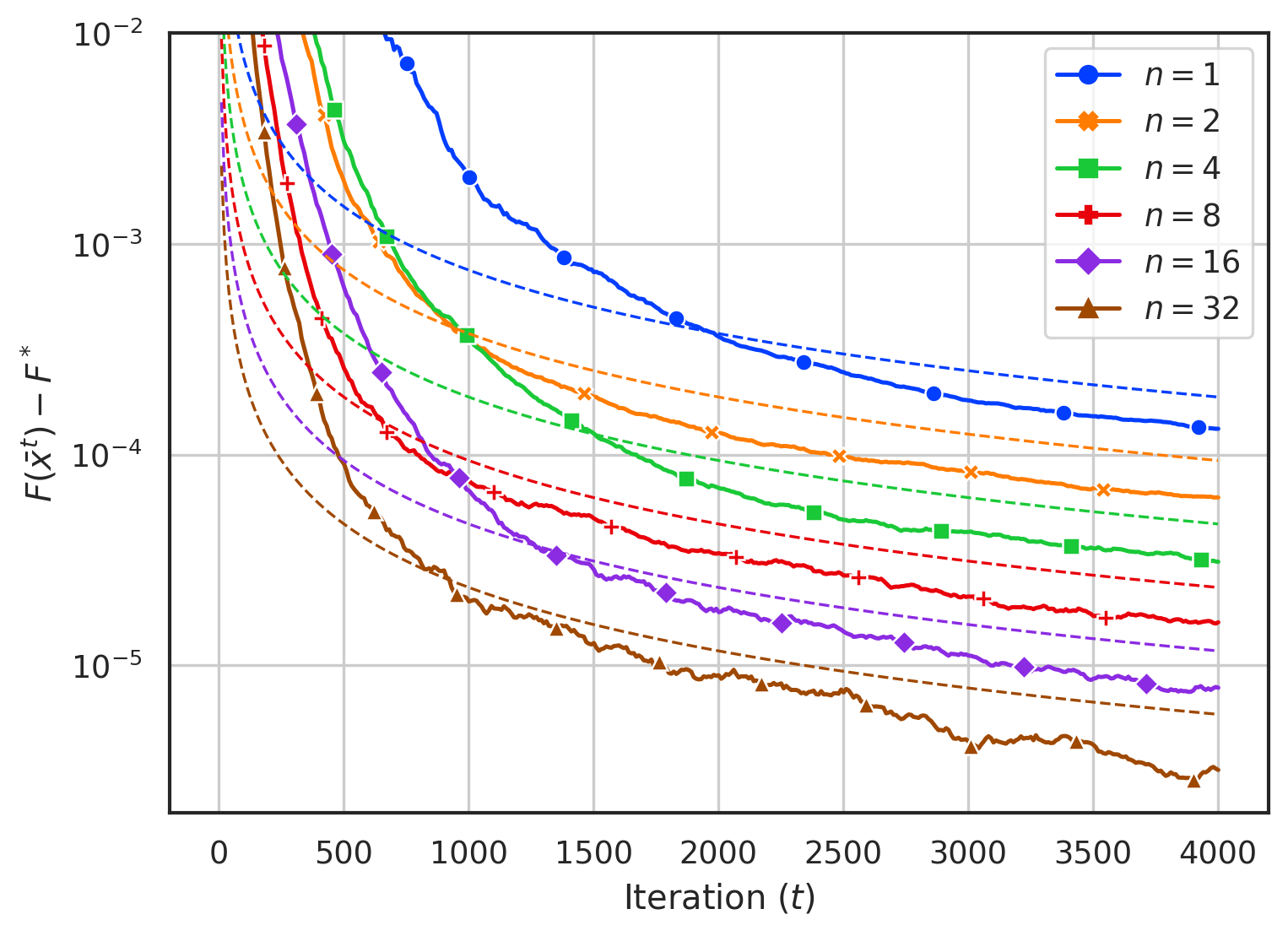}
		\caption{Speed-up in the network size.}
	\end{subfigure}
	\caption{Local SGD with different communication strategies with $F(\bx)=\E_\zeta f(\bx, \zeta)$ defined in \eqref{eq: f-zeta}, $c=9, \sigma=0.5, d=3, \beta=1$. Figures (a) and (b) show the error of different communication methods over iteration and communication round, respectively, with a fixed network size of $n=20$. Figure (c) shows the convergence of Local SGD with the communication method proposed in this paper ($n$ communication rounds) for different network sizes. The dashed lines are showing $\sigma^2/(\mu n t)$.}
	\label{fig: res2}
\end{figure}

Figure \ref{fig: res2}(a) shows that {\em the method with increasing communication intervals ($H_i = 3(i+1)$) proposed in this paper performs better than all the other communication strategies in the transient time as well as in the final error}, requiring much less communication rounds. In particular, the method with the same number of communications but fixed intervals ($H=25$), has both higher transient error and final error. This affirms the advantages of having more frequent communication at the beginning of the optimization. Indeed, observe that in Figures \ref{fig: res2}(a), the only method which outperforms the method we propose is the one that communicates at every step. 

Figure \ref{fig: res2}(b) reveals the effectiveness of each communication round in different methods. We observe that there's an initial spike in the initial communications in methods $H=5$ and $H_i = 3(i+1)$. This is mainly because these two methods have more frequent communications at the beginning of the training, where the step-sizes are larger. Other methods experience this increase as well, however since they communicate later, it's not observed in this figure. Indeed, observe that the only method which makes better use of communication periods than our method in Figure \ref{fig: res2}(b) is one-shot averaging, which is not competitive in terms of its final error.

Figure \ref{fig: res2}(c) verifies that linear-speed up in the number of workers can be achieved with only $R=n$ communication rounds. Moreover, it shows that Local SGD achieves the optimal convergence rate of $\sigma^2/(n \mu T)$ asymptotically.

\subsection{Regularized Logistic Regression}
We also performed additional numerical experiments with regularized logistic regression using two real data sets. Due to space constraints, the results are presented in supplementary information.

\section{Conclusion}\label{sec: conclustion}
We have presented a new  analysis of Local SGD and studied the effect of choice of communication times on the final optimality error. We proposed a communication strategy which achieves linear speed-up in the number of workers with only $\Omega(n)$ communication rounds, independent of the total number of iterations $T$.
Numerical experiments further confirmed our theoretical findings, and showed that our method achieves smaller error than previous methods using fewer communications.

\section*{Broader Impact}

The results presented in this paper could help speed up training in many machine learning applications. The potential broader impacts are therefore somewhat generic for machine learning: this research could amplify  all the benefits ML can bring by making it cheaper in terms of computational cost, while simultaneously amplifying all the ways ML could be misused. 
%Authors are required to include a statement of the broader impact of their work, including its ethical aspects and future societal consequences. Authors should discuss both positive and negative outcomes, if any. For instance, authors should discuss a) who may benefit from this research, b) who may be put at disadvantage from this research, c) what are the consequences of failure of the system, and d) whether the task/method leverages biases in the data. If authors believe this is not applicable to them, authors can simply state this.

% Use unnumbered first level headings for this section, which should go at the end of the paper. {\bf Note that this section does not count towards the eight pages of content that are allowed.}

\begin{ack}
% Use unnumbered first level headings for the acknowledgments. All acknowledgments go at the end of the paper before the list of references. Moreover, you are required to declare  funding (financial activities supporting the submitted work) and competing interests (related financial activities outside the submitted work). More information about this disclosure can be found at: \url{https://neurips.cc/Conferences/2020/PaperInformation/FundingDisclosure}.

%Do {\bf not} include this section in the anonymized submission, only in the final paper. You can use the \texttt{ack} environment provided in the style file to autmoatically hide this section in the anonymized submission.
\end{ack}

\small
\bibliographystyle{alpha} %{apalike}
\bibliography{References}

\newcommand{\etalchar}[1]{$^{#1}$}
\begin{thebibliography}{ZDSMR16}

\bibitem[AGL{\etalchar{+}}17]{alistarh2017qsgd}
Dan Alistarh, Demjan Grubic, Jerry Li, Ryota Tomioka, and Milan Vojnovic.
\newblock Qsgd: Communication-efficient sgd via gradient quantization and
  encoding.
\newblock In {\em Advances in Neural Information Processing Systems}, pages
  1709--1720, 2017.

\bibitem[AS15]{arjevani2015communication}
Yossi Arjevani and Ohad Shamir.
\newblock Communication complexity of distributed convex learning and
  optimization.
\newblock In {\em Advances in neural information processing systems}, pages
  1756--1764, 2015.

\bibitem[BSS16]{bijral2016data}
Avleen~S Bijral, Anand~D Sarwate, and Nathan Srebro.
\newblock On data dependence in distributed stochastic optimization.
\newblock {\em arXiv preprint arXiv:1603.04379}, 2016.

\bibitem[CL11]{CC01a}
Chih-Chung Chang and Chih-Jen Lin.
\newblock {LIBSVM}: A library for support vector machines.
\newblock {\em ACM Transactions on Intelligent Systems and Technology},
  2:27:1--27:27, 2011.
\newblock Software available at \url{http://www.csie.ntu.edu.tw/~cjlin/libsvm}.

\bibitem[DGBSX12]{dekel2012optimal}
Ofer Dekel, Ran Gilad-Bachrach, Ohad Shamir, and Lin Xiao.
\newblock Optimal distributed online prediction using mini-batches.
\newblock {\em Journal of Machine Learning Research}, 13(Jan):165--202, 2012.

\bibitem[DP19]{dieuleveut2019communication}
Aymeric Dieuleveut and Kumar~Kshitij Patel.
\newblock Communication trade-offs for local-sgd with large step size.
\newblock In {\em Advances in Neural Information Processing Systems}, pages
  13579--13590, 2019.

\bibitem[GBS20]{godichon2017rates}
Antoine Godichon-Baggioni and Sofiane Saadane.
\newblock On the rates of convergence of parallelized averaged stochastic
  gradient algorithms.
\newblock {\em Statistics}, pages 1--18, 2020.

\bibitem[HBM19]{hendrikx2019accelerated}
Hadrien Hendrikx, Francis Bach, and Laurent Massouli{\'e}.
\newblock An accelerated decentralized stochastic proximal algorithm for finite
  sums.
\newblock In {\em Advances in Neural Information Processing Systems}, pages
  952--962, 2019.

\bibitem[HKMC19]{haddadpour2019local}
Farzin Haddadpour, Mohammad~Mahdi Kamani, Mehrdad Mahdavi, and Viveck Cadambe.
\newblock Local sgd with periodic averaging: Tighter analysis and adaptive
  synchronization.
\newblock In {\em Advances in Neural Information Processing Systems}, pages
  11080--11092, 2019.

\bibitem[KLB{\etalchar{+}}20]{koloskova2020unified}
Anastasia Koloskova, Nicolas Loizou, Sadra Boreiri, Martin Jaggi, and
  Sebastian~U Stich.
\newblock A unified theory of decentralized sgd with changing topology and
  local updates.
\newblock {\em arXiv preprint arXiv:2003.10422}, 2020.

\bibitem[KMA{\etalchar{+}}19]{kairouz2019advances}
Peter Kairouz, H~Brendan McMahan, Brendan Avent, Aur{\'e}lien Bellet, Mehdi
  Bennis, Arjun~Nitin Bhagoji, Keith Bonawitz, Zachary Charles, Graham Cormode,
  Rachel Cummings, et~al.
\newblock Advances and open problems in federated learning.
\newblock {\em arXiv preprint arXiv:1912.04977}, 2019.

\bibitem[KMR20]{khaled2019tighter}
A~Khaled, K~Mishchenko, and P~Richt{\'a}rik.
\newblock Tighter theory for local sgd on identical and heterogeneous data.
\newblock In {\em The 23rd International Conference on Artificial Intelligence
  and Statistics (AISTATS 2020)}, 2020.

\bibitem[KMY{\etalchar{+}}16]{konevcny2016federated}
Jakub Kone{\v{c}}n{\`y}, H~Brendan McMahan, Felix~X Yu, Peter Richt{\'a}rik,
  Ananda~Theertha Suresh, and Dave Bacon.
\newblock Federated learning: Strategies for improving communication
  efficiency.
\newblock {\em arXiv preprint arXiv:1610.05492}, 2016.

\bibitem[LHM{\etalchar{+}}18]{lin2017deep}
Yujun Lin, Song Han, Huizi Mao, Yu~Wang, and Bill Dally.
\newblock Deep gradient compression: Reducing the communication bandwidth for
  distributed training.
\newblock In {\em International Conference on Learning Representations}, 2018.

\bibitem[LSPJ18]{lin2018don}
Tao Lin, Sebastian~U Stich, Kumar~Kshitij Patel, and Martin Jaggi.
\newblock Don't use large mini-batches, use local sgd.
\newblock {\em arXiv preprint arXiv:1808.07217}, 2018.

\bibitem[MHM10]{mcdonald2010distributed}
Ryan McDonald, Keith Hall, and Gideon Mann.
\newblock Distributed training strategies for the structured perceptron.
\newblock In {\em Human language technologies: The 2010 annual conference of
  the North American chapter of the association for computational linguistics},
  pages 456--464. Association for Computational Linguistics, 2010.

\bibitem[MMR{\etalchar{+}}17]{mcmahan2016communication}
Brendan McMahan, Eider Moore, Daniel Ramage, Seth Hampson, and Blaise~Aguera
  y~Arcas.
\newblock Communication-efficient learning of deep networks from decentralized
  data.
\newblock In {\em Artificial Intelligence and Statistics}, pages 1273--1282,
  2017.

\bibitem[MMS{\etalchar{+}}09]{mcdonald2009efficient}
Ryan Mcdonald, Mehryar Mohri, Nathan Silberman, Dan Walker, and Gideon~S Mann.
\newblock Efficient large-scale distributed training of conditional maximum
  entropy models.
\newblock In {\em Advances in neural information processing systems}, pages
  1231--1239, 2009.

\bibitem[RN16]{rosenblatt2016optimality}
Jonathan~D Rosenblatt and Boaz Nadler.
\newblock On the optimality of averaging in distributed statistical learning.
\newblock {\em Information and Inference: A Journal of the IMA}, 5(4):379--404,
  2016.

\bibitem[SCJ18]{stich2018sparsified}
Sebastian~U Stich, Jean-Baptiste Cordonnier, and Martin Jaggi.
\newblock Sparsified sgd with memory.
\newblock In {\em Advances in Neural Information Processing Systems}, pages
  4447--4458, 2018.

\bibitem[SK19]{stich2019error}
Sebastian~U Stich and Sai~Praneeth Karimireddy.
\newblock The error-feedback framework: Better rates for sgd with delayed
  gradients and compressed communication.
\newblock {\em arXiv preprint arXiv:1909.05350}, 2019.

\bibitem[SOP20]{olshevsky2018robust}
Artin Spiridonoff, Alex Olshevsky, and Ioannis~Ch Paschalidis.
\newblock Robust asynchronous stochastic gradient-push: asymptotically optimal
  and network-independent performance for strongly convex functions.
\newblock {\em Journal of Machine Learning Research}, 2020.

\bibitem[SR13]{schmidt2013fast}
Mark Schmidt and Nicolas~Le Roux.
\newblock Fast convergence of stochastic gradient descent under a strong growth
  condition.
\newblock {\em arXiv preprint arXiv:1308.6370}, 2013.

\bibitem[Sti19]{stich2018local}
Sebastian~U. Stich.
\newblock Local {SGD} converges fast and communicates little.
\newblock In {\em International Conference on Learning Representations}, 2019.

\bibitem[TSC{\etalchar{+}}20]{tang2020communication}
Zhenheng Tang, Shaohuai Shi, Xiaowen Chu, Wei Wang, and Bo~Li.
\newblock Communication-efficient distributed deep learning: A comprehensive
  survey.
\newblock {\em arXiv preprint arXiv:2003.06307}, 2020.

\bibitem[WJ18a]{wang2018adaptive}
Jianyu Wang and Gauri Joshi.
\newblock Adaptive communication strategies to achieve the best error-runtime
  trade-off in local-update sgd.
\newblock {\em Systems for ML}, 2018.

\bibitem[WJ18b]{wang2018cooperative}
Jianyu Wang and Gauri Joshi.
\newblock Cooperative sgd: A unified framework for the design and analysis of
  communication-efficient sgd algorithms.
\newblock {\em arXiv preprint arXiv:1808.07576}, 2018.

\bibitem[WPS{\etalchar{+}}20]{woodworth2020local}
Blake Woodworth, Kumar~Kshitij Patel, Sebastian~U Stich, Zhen Dai, Brian
  Bullins, H~Brendan McMahan, Ohad Shamir, and Nathan Srebro.
\newblock Is local sgd better than minibatch sgd?
\newblock {\em arXiv preprint arXiv:2002.07839}, 2020.

\bibitem[WSY{\etalchar{+}}19]{wang2019matcha}
Jianyu Wang, Anit~Kumar Sahu, Zhouyi Yang, Gauri Joshi, and Soummya Kar.
\newblock Matcha: Speeding up decentralized sgd via matching decomposition
  sampling.
\newblock {\em arXiv preprint arXiv:1905.09435}, 2019.

\bibitem[WWLZ18]{wangni2018gradient}
Jianqiao Wangni, Jialei Wang, Ji~Liu, and Tong Zhang.
\newblock Gradient sparsification for communication-efficient distributed
  optimization.
\newblock In {\em Advances in Neural Information Processing Systems}, pages
  1299--1309, 2018.

\bibitem[YYZ19]{yu2019parallel}
Hao Yu, Sen Yang, and Shenghuo Zhu.
\newblock Parallel restarted sgd with faster convergence and less
  communication: Demystifying why model averaging works for deep learning.
\newblock In {\em Proceedings of the AAAI Conference on Artificial
  Intelligence}, volume~33, pages 5693--5700, 2019.

\bibitem[ZC18]{zhou2017convergence}
Fan Zhou and Guojing Cong.
\newblock On the convergence properties of a k-step averaging stochastic
  gradient descent algorithm for nonconvex optimization.
\newblock In {\em Proceedings of the 27th International Joint Conference on
  Artificial Intelligence}, pages 3219--3227. AAAI Press, 2018.

\bibitem[ZCL15]{zhang2015deep}
Sixin Zhang, Anna~E Choromanska, and Yann LeCun.
\newblock Deep learning with elastic averaging sgd.
\newblock In {\em Advances in neural information processing systems}, pages
  685--693, 2015.

\bibitem[ZDJW13]{zhang2013information}
Yuchen Zhang, John Duchi, Michael~I Jordan, and Martin~J Wainwright.
\newblock Information-theoretic lower bounds for distributed statistical
  estimation with communication constraints.
\newblock In {\em Advances in Neural Information Processing Systems}, pages
  2328--2336, 2013.

\bibitem[ZDSMR16]{zhang2016parallel}
Jian Zhang, Christopher De~Sa, Ioannis Mitliagkas, and Christopher R{\'e}.
\newblock Parallel sgd: When does averaging help?
\newblock {\em arXiv preprint arXiv:1606.07365}, 2016.

\bibitem[ZDW13]{zhang2013communication}
Yuchen Zhang, John~C Duchi, and Martin~J Wainwright.
\newblock Communication-efficient algorithms for statistical optimization.
\newblock {\em Journal of Machine Learning Research}, 14(1):3321--3363, 2013.

\bibitem[ZWLS10]{zinkevich2010parallelized}
Martin Zinkevich, Markus Weimer, Lihong Li, and Alex~J Smola.
\newblock Parallelized stochastic gradient descent.
\newblock In {\em Advances in neural information processing systems}, pages
  2595--2603, 2010.

\end{thebibliography}

\newpage
\appendix
\section{Missing Proofs}
Let us define the following notations used in the proofs presented here.
\begin{align*}
    \bbg^t := (\sum_{i=1}^n \bg_i^t)/n, \qquad G^t := \frac{1}{n}\sum_{i=1}^n \Vert \bg_i^t \Vert^2, \qquad & \e_i^t := \hbg_i^t - \bg_i^t.
\end{align*}
Moreover, define $\F^t:=\{\bx_i^k, \hbg_i^k | 1\leq i \leq n, 0\leq k \leq t-1 \} \cup \{\bx_i^t | 1 \leq i \leq n\}$.

\begin{lemmapp}[\ref{lem: error decay}]%\label{lem: error decay}
    Let Assumptions \ref{as: F} and \ref{as: g} hold. Then,
    \begin{align*}
	    \xi^{t+1} \leq \xi^t(1 - \mu \eta_t) + \frac{L^2 \eta_t}{2n} \E \left[ \sum_{i=1}^n \Vert \bbx^t - \bx_i^t \Vert^2 \right] + \frac{\eta_t^2 L}{2} \E[\Vert \tbg^t \Vert_2^2 ] 
	    - \frac{\eta_t}{2n} \E \left[\sum_{i=1}^n \Vert \nabla F(\bx_i^t) \Vert^2\right].
	\end{align*}
\end{lemmapp}

\begin{proof}[Proof of Lemma \ref{lem: error decay}]
% 	Let us define $\tbg^t:= (\sum_{i =1}^n \hbg_i^t)/n$.
	By Assumption \ref{as: F} and \eqref{eq: average x update} we have,
	\begin{align}\label{eq: opt E1}
	\E[F(\bbx^{t+1}) - F(\bbx^t)] \leq -\eta_t \E[\langle \nabla F(\bbx^t), \tbg^t \rangle ] + \frac{\eta_t^2 L}{2} \E[\Vert \tbg^t \Vert_2^2 ].
	\end{align}
	We bound the first term on the R.H.S of \eqref{eq: opt E1} by conditioning on $\F^t$ as follows:
	\begin{align} \label{eq: opt E2}
	\E[\langle  \nabla F(\bbx^t), \tbg^t \rangle | \mathcal{F}^{t} ] &= \frac{1}{n} \sum_{i=1}^n \langle \nabla F(\bbx^t), \E[\hbg_i^t | \bx_i^t] \rangle \nonumber \\
	&= \frac{1}{2} \Vert \nabla F(\bbx^t) \Vert^2 + \frac{1}{2n} \sum_{i=1}^n \Vert \nabla F(\bx_i^t) \Vert^2 - \frac{1}{2n} \sum_{i=1}^n \Vert \nabla F(\bbx^t) - \nabla F(\bx_i^t) \Vert^2  \nonumber\\
	& \geq \mu (F(\bbx^t) - F^*) + \frac{1}{2n} \sum_{i=1}^n \Vert \nabla F(\bx_i^t) \Vert^2  - \frac{L^2}{2n} \sum_{i=1}^n \Vert \bbx^t - \bx_i^t \Vert^2,
	\end{align}
	where we used $\langle a, b \rangle  = \frac{1}{2} \Vert a \Vert ^2 + \frac{1}{2} \Vert b\Vert^2 - \frac{1}{2} \Vert a-b \Vert ^2$ in the second equation and $(1/2)\Vert \nabla F(\bx) \Vert^2 \geq \mu (F(\bx) - F^*)$ as well as smoothness of $F$ in the last inequality.
	Taking full expectation of \eqref{eq: opt E2} and combining it with \eqref{eq: opt E1} concludes the lemma.
% 	\begin{align*}
% 	    \xi^{t+1} \leq \xi^t(1 - \mu \eta_t) + \frac{L^2 \eta_t}{2n} \E \left[ \sum_{i=1}^n \Vert \bbx^t - \bx_i^t \Vert^2 \right] + \frac{\eta_t^2 L}{2} \E[\Vert \tbg^t \Vert_2^2 ] 
% 	    - \frac{\eta_t}{2n} \sum_{i=1}^n \Vert \nabla F(\bx_i^t) \Vert^2.
% 	\end{align*}
\end{proof}

We state an important identity in the following lemma.
\begin{lemma}\label{lem: u - ubar}
    Let $\bu_1, \ldots \bu_n \in \R^d$ be $n$ arbitrary vectors. Define $\bar \bu = (\sum_{i=1}^n \bu_i)/n$. Then,
    \begin{align*}
        \sum_{i=1}^n \Vert \bu_i - \bar \bu \Vert^2 = \sum_{i=1}^n \Vert \bu_i \Vert^2 - n \Vert \bar \bu \Vert^2.
    \end{align*}
\end{lemma}
% The proof of this lemma involves simple algebraic manipulations and therefore is skipped.
\begin{proof}
We have
    \begin{align*}
        \sum_{i=1}^n \Vert \bu_i - \bar \bu \Vert^2 &= \sum_{i=1}^n \Vert \bu_i \Vert^2 + n \Vert \bar \bu \Vert^2 - 2\sum_{i=1}^n \langle \bu_i, \bar \bu \rangle \\
        &= \sum_{i=1}^n \Vert \bu_i \Vert^2 + n \Vert \bar \bu \Vert^2 - 2n \langle \bar \bu, \bar \bu \rangle \\
        & = \sum_{i=1}^n \Vert \bu_i \Vert^2 - n \Vert \bar \bu \Vert^2.
    \end{align*}
\end{proof}

% The following lemmas, bounds the consensus error.
\begin{lemmapp}[\ref{lem: consensus 1}] %\label{lem: consensus 1}
	Let Assumptions \ref{as: F} and \ref{as: g} hold. Then,
	\begin{multline*}%\label{eq: consensus 2}
	\E\left[ \sum_{i=1}^n \Vert \bx_i^{t+1} - \bbx^{t+1} \Vert^2 \right] \leq \E\left[ \sum_{i=1}^n \Vert \bx_i^{t} - \bbx^t \Vert^2 \right](1 - 2\eta_t \mu + \eta_t^2 L^2) \\ + (n-1)\eta_t^2\sigma^2 + (1-\frac{1}{n})\eta_t^2 c\E [\sum_{i=1}^n \Vert g_i^t \Vert^2].
	\end{multline*}
\end{lemmapp}
\begin{proof}[Proof of Lemma \ref{lem: consensus 1}]
    % We have $\bx_i^{t+1} = \bx_i^t - \eta_t(\bg_i^t + \e_i^t)$. Conditioned on $\mathcal F_t$, taking expectations of both sides implies,
    % \begin{align*}
        % \E[\bx_i^{t+1} | \mathcal{F}_t] = \bx_i^t - \eta \bg_i^t.
    % \end{align*}
	We have,
	\begin{align}\label{eq: consensus var}
	\E [  \sum_{i=1}^{n} \Vert \bx_i^{t+1} - \bbx^{t+1} \Vert^2 ] = \sum_{i=1}^{n} \Vert \E[\bx_i^{t+1} - \bbx^{t+1}] \Vert^2 + \sum_{i=1}^{n} \E \left[\Vert \bx_i^{t+1} - \bbx^{t+1} - \E[\bx_i^{t+1} - \bbx^{t+1}] \Vert^2 \right].
	\end{align}
	Let us consider the first term on the right hand side of \eqref{eq: consensus var}. Taking conditional expectation of both sides of \eqref{eq: average x update} implies,
	\begin{align} \label{eq: consensus 1}
	\sum_{i=1 }^{n} \Vert \E[\bx_i^{t+1} - \bbx^{t+1} | \ \F^{t}] \Vert^2 &= \sum_{i =1}^{n} \Vert \bx_i^t - \bbx^t - \eta_t(\bg_i^t - \bar \bg^t) \Vert^2 
	\nonumber \\
	&= \sum_{i=1}^{n} \left(\Vert \bx_i^t - \bbx^t \Vert^2 + \eta_t^2\Vert \bg_i^t - \bar \bg^t \Vert^2 - 2 \eta_t \langle \bg_i^t, \bx_i^t - \bbx^t \rangle \right)
	\end{align}
	By $L$-smoothness of $F$,
	\begin{align} \label{eq: consensus L}
	\sum_{i=1}^n \Vert \bg_i^t - \bar \bg^t \Vert^2 = \frac{1}{n} \sum_{\{i,j\}} \Vert \bg_i^t - \bg_j^t \Vert^2 \leq \frac{L^2}{n} \sum_{\{i,j\}} \Vert \bx_i^t - \bx_j^t \Vert^2 = L^2 \sum_{i=1}^n \Vert \bx_i^t - \bbx^t \Vert^2.
	\end{align}
	Moreover, by $\mu$-strong convexity of $F$,
	\begin{multline} \label{eq: consensus mu}
	\sum_{i=1}^n \langle \bg_i^t, \bx_i^t - \bbx^t \rangle = \sum_{i=1}^n  \langle \bg_i^t, \frac{1}{n} \sum_{j=1}^n (\bx_i^t - \bx_j^t) \rangle \\
	= \frac{1}{n} \sum_{\{i,j\}} \langle \bg_i^t - \bg_j^t, \bx_i^t - \bx_j^t \rangle \geq \frac{\mu}{n} \sum_{\{i,j\}} \Vert \bx_i - \bx_j \Vert^2 = \mu \sum_{i=1}^{n} \Vert \bx_i^t - \bbx^t \Vert^2,
	\end{multline}
	where we used $\langle \nabla F(\bx) - \nabla F(\by), \bx - \by \rangle \geq \mu \Vert \bx - \by \Vert^2$ in the inequality.
	Combining \eqref{eq: consensus 1}-\eqref{eq: consensus mu} we obtain,
	\begin{align*}
	\sum_{i=1 }^{n} \Vert \E[\bx_i^{t+1} - \bbx^{t+1} | \F^t] \Vert^2 &\leq \sum_{i=1}^{n} \Vert \bx_i^t - \bbx^t \Vert^2 \left( 1 - 2\eta_t \mu + \eta_t^2L^2 \right), 
	\end{align*}
	Now, consider the second term on the right hand side of \eqref{eq: consensus var}.
	We have,
	\begin{align*}
	    \sum_{i=1}^{n} \E \left[\left\Vert \bx_i^{t+1} - \bbx^{t+1} - \E[\bx_i^{t+1} - \bbx^{t+1}] \right\Vert^2 | \F^t \right] &= 
	    \sum_{i=1}^{n} \E \left[\left\Vert \bx_i^{t+1} - \E[\bx_i^{t+1}] - (\bbx^{t+1} - \E[\bbx^{t+1}]) \right\Vert^2 | \F^t \right] \\
	    &= \eta_t^2 \sum_{i=1}^{n} \E \left[ \left\Vert \e_i^t - \bar \e^t \right\Vert^2 | \F^t \right] \\
	    &= \eta_t^2 \left(\sum_{i=1}^{n} \E \left[ \left\Vert \e_i^t \right\Vert^2 | \F^t \right]- n \E \left[ \left\Vert \bar \e^t \right\Vert^2 |\F^t \right]\right) \\
	    &= \eta_t^2 \sum_{i=1}^{n} \E \left[ \left\Vert \e_i^t \right\Vert^2 | \F^t \right](1-\frac{1}{n}) \\
	    &\leq (n-1)\eta_t^2\sigma^2 + (1-\frac{1}{n})\eta_t^2 c\sum_{i=1}^n \Vert g_i^t \Vert^2,
	\end{align*}
	where $\e_i^t$ are defined at the beginning of this section and $\bar \e^t := (\sum_{i=1}^n \e_i^t)/n$ and we used Lemma \ref{lem: u - ubar} in the third equation and the conditional independence of $\e_i^t$ to use $\E[\Vert \bar \e^t \Vert^2 | \F^t] = (1/n^2)\sum_{i=1}^n \E[\Vert \e_i^t \Vert^2 | \F^t]$ in the last equality.
	%Using Lemma \ref{lem: variance} and $\text{var}(\bx_i^{t+1}|\mathcal{F}_t) \leq \eta_t^2\sigma^2$ implies
	Taking full expectation of the two relations above with respect to $\F^t$ and combining them with \eqref{eq: consensus var} completes the proof.
\end{proof}

% Set $\tau(t) = \max\{t' \in \mathcal{I}| t' \leq t\}$ as the most recent time workers have reached consensus.
% and $t = \tau(t) + r$ for some $0\leq r \leq \tau -1$.
\begin{lemmapp}[\ref{lem: consensus 2}]%\label{lem: consensus 2}
	Let assumptions of Theorem \ref{thm: general} hold. Then,
	\begin{align*}
	\E\left[ \sum_{i=1}^n \Vert \bx_i^{t} - \bbx^t \Vert^2 \right] \leq
% 	(n-1)\sigma^2\frac{9(t - \tau(t))}{\mu^2 (t+\beta)^2} + 
    9(n-1)\sum_{k=\tau(t)}^{t-1}\frac{c\E[G^k]+\sigma^2}{\mu^2(t+\beta)^2}.
	\end{align*}
\end{lemmapp}

Before proving this lemma, let us state and prove the following lemma.
\begin{lemma}\label{lem: products}
	Let $b \geq a > 2$ be integers. Define $\Phi(a,b) = \prod_{i=a}^b \left( 1 - \frac{2}{i} \right)$. We then have $
	\Phi(a,b) \leq \left( \frac{a}{b+1} \right)^{2}.$
\end{lemma}
\begin{proof}
	Indeed,
	\begin{align*}
	\ln (\Phi(a,b))  =  \sum_{i=a}^b \ln \left( 1 - \frac{2}{i} \right)   
	\leq  \sum_{i=a}^b - \frac{2}{i} 
	\leq  -  2\left[ \ln (b+1) - \ln (a) \right].
	\end{align*}
	where we used the inequality $\ln (1-x) \leq -x$ as well as the standard technique of viewing $\sum_{i=a}^b 1/i$ as a Riemann sum for $\int_{a}^{b+1} 1/x ~dx$ and observing that the Riemann sum overstates the integral. Exponentiating both sides now implies the lemma.
\end{proof}

\begin{proof}[Proof of Lemma \ref{lem: consensus 2}]
	Define $a^k =  \E\left[ \sum_{i=1}^n \Vert \bx_i^{k} - \bbx^k \Vert^2 \right]$ and $\Delta_k = (1 - 2\eta_k \mu + \eta_k^2 L^2)$ for $k\geq 0$ . By Lemma \ref{lem: consensus 1},
	\begin{align*}
	a^{t} &\leq \Delta_{t-1}a^{t-1}  + \eta_{t-1}^2(n-1)(\sigma^2 + c \E[G^{t-1}]) \\
	& \leq \Delta_{t-1}(\Delta_{t-2}a^{t-2} + \eta_{t-2}^2(n-1)(\sigma^2+c\E[G^{t-2}])) + \eta_{t-1}^2(n-1)(\sigma^2+c\E[G^{t-1}]) \\
	& \leq \ldots \leq \prod_{k=\tau(t)}^{t-1} \Delta_k a^{\tau(t)} + 
	(n-1)\sum_{k=\tau(t)}^{t-1} \eta_k^2(\sigma^2+c\E[G^k])\prod_{i=k+1}^{t-1} \Delta_i \\
	&= (n-1)\sum_{k=\tau(t)}^{t-1} \eta_k^2(\sigma^2+c\E[G^k])\prod_{i=k+1}^{t-1} \Delta_i,
	\end{align*}
	where we used $a^{\tau(t)}= 0$ in the last equation. 
	% Let $\eta_k = 2/ \mu (k+\beta)$ and $\beta \geq 2\kappa^2 $ where $\kappa = L/\mu$ is the condition number of $F$.
	By the choice of stepsize and $\beta\geq 2\kappa^2$,  we have %$\Delta_k \leq (1 - 2/(k+\beta))$.
	\begin{align*}
	    \Delta_k = 1-\frac{4}{(k+\beta)} + \frac{4L^2}{\mu^2 (k+\beta)^2} \leq 1 - \frac{4}{k+\beta} + \frac{4\kappa^2}{(k+\beta)\beta}
	    \leq 1 - \frac{4}{k+\beta} + \frac{2}{(k+\beta)} =  1 - \frac{2}{k+\beta}.
	\end{align*}
	Therefore, by Lemma \ref{lem: products}, 
	\begin{align*}
	a^t \leq (n-1) \sum_{k = \tau(t)}^{t-1} \frac{4(\sigma^2+c\E[G^k])}{\mu^2 (k+\beta)^2 } \frac{(k+\beta+1)^2}{(t + \beta)^2} \leq (n-1) \sum_{k = \tau(t)}^{t-1} \frac{9(\sigma^2+c\E[G^k])}{\mu^2 (t+\beta)^2},
% 	= (n-1)\sigma^2\frac{9(t - \tau(t))}{\mu^2 (t+\beta)^2}.
	\end{align*}
	where we used $(k+\beta + 1)/(k+\beta)\leq (\beta+1)/\beta \leq 3/2$ since $\beta \geq 2\kappa^2 \geq 2$.
\end{proof}

\begin{proof}[Proof of Lemma \ref{lem: variance}]
    We have,
    \begin{align*}
        \E[\Vert \tbg^t \Vert^2|\F^t] = \E[\Vert \bbg^t + \bar \e^t \Vert^2 | \F_t] = \Vert \bbg^t \Vert^2 + \E[\Vert \bar \e^t \Vert^2| \F^t] \leq  \frac{1}{n}\sum_{i=1}^n \Vert \bg_i^t \Vert^2 + \frac{1}{n^2}\sum_{i=1}^n(\sigma^2 + c\Vert \bg_i^t \Vert^2),
    \end{align*}
    where in the last inequality we used Lemma \ref{lem: u - ubar} and the conditional independency of $\e_i^t$ to separate the noise terms.
\end{proof}

\begin{proof}[Proof of Theorem \ref{thm: general}]
% 	Also, $\E [ \Vert \tbg^t \Vert_2^2 ] \leq \sigma^2/n$.
% 	Defining $\xi^t:= \E[F(\bbx^t)] - F^*$, 
	Combining Equations Lemmas \ref{lem: error decay}-\ref{lem: variance} and plugging $\eta_t = 2/(\mu(t+\beta))$ we obtain
	\begin{multline*}
	    \xi^{t+1} \leq \xi^t(1 - \mu \eta_t) + 
	    \frac{9L^2}{\mu^3(t+\beta)^3} 
	    \sum_{k=\tau(t)}^{t-1} (c\E[G^k]+\sigma^2) 
	    \\
	    + \frac{2L}{\mu^2(t+\beta)^2} \left( (1+\frac{c}{n}) \E[G^t] + \frac{\sigma^2}{n} \right)
	    - \frac{1}{\mu(t+\beta)} \E[G^t].
	\end{multline*}
% 	\begin{align*}
% 	\xi^{t+1} \leq \xi^t(1 - \mu\eta_t) + \frac{\eta_t^2  L\sigma^2}{2n}+ \frac{9L^2(n-1) \sigma^2}{8n}\eta_t^3(t-\tau(t)) + \frac{L\eta_t^2 - \eta_t}{2n} \sum_{i=1}^n \E[\Vert \nabla F(\bx_i^t) \Vert^2].
% 	\end{align*}
% 	Notice that by the choice of $\beta>2\kappa^2$, we have $L\eta_t^2 - \eta_t<0$ automatically satisfied and we can simply ignore the last term in equation above.
	Let us multiply both sides of relation above by $(t+\beta)^2$ 
	%using inequality $(1-\mu \eta_t)(t+\beta)^2 \leq (t+\beta - 1)^2$ we have
	and use the following inequality
	\begin{align*}
	    (1-\mu \eta_t)(t+\beta)^2 = (1 - \frac{2}{t+\beta})(t+\beta)^2 = (t+\beta)^2 - 2(t+\beta) < (t+\beta-1)^2,
	\end{align*}
	to obtain,
	\begin{multline*}
	\xi^{t+1} (t+\beta)^2 \leq \xi^t (t+\beta-1)^2 + \frac{2L \sigma^2}{n \mu^2} + 
	\frac{9L^2}{\mu^3(t+\beta)} \sum_{k=\tau(t)}^{t-1} (c\E[G^k]+\sigma^2) \\ + 
	\left( \frac{2L}{\mu^2}(1+\frac{c}{n}) - \frac{t+\beta}{\mu} \right)\E[G^t].
	\end{multline*}
	Summing relation above for $t=\tau_i,\ldots,\tau_{i+1}-1$, where $\tau_i,\tau_{i+1}\in \I$ are two consecutive communication times, implies,
	\begin{multline*}%\label{eq: opt E3}
	\xi^{\tau_{i+1}}(\tau_{i+1} + \beta -1 )^2 \leq \xi^{\tau_i}(\tau_{i} + \beta -1 )^2  +  \frac{2L \sigma^2}{n \mu^2}(\tau_{i+1} - \tau_i)
	+ \frac{9L^2\sigma^2}{\mu^3}\sum_{t=\tau_i}^{\tau_{i+1}-1} \frac{t-\tau_i}{t+\beta} \\
	+ \sum_{t=\tau_i}^{\tau_{i+1}-1}\E[G^t] \left( \sum_{k=t+1}^{\tau_{i+1}-1}\frac{9L^2 c}{\mu^3(k+\beta)} + \frac{2L}{\mu^2}(1+\frac{c}{n}) - \frac{t+\beta}{\mu} \right).
	\end{multline*}
	Each of the coefficients of $\E[G^t]$ in above can be bounded by,
	\begin{align*}
	    \sum_{k=t+1}^{\tau_{i+1}-1}\frac{9L^2 c}{\mu^3(k+\beta)} + \frac{2L}{\mu^2}(1+\frac{c}{n}) - \frac{t+\beta}{\mu} &\leq 
	    \frac{9L^2c}{\mu^3}\ln(\frac{\tau_{i+1}+\beta -1 }{\tau_i + \beta})
	    + \frac{2L}{\mu^2}(1+\frac{c}{n}) - \frac{\tau_i + 1 +\beta}{\mu} \\
	    & = \frac{1}{\mu}\left( 9\kappa^2c \ln(1 + \frac{H_i - 1}{\tau_i + \beta}) + 2\kappa(1+\frac{c}{n}) - (\tau_i +1 + \beta)\right) \\
	    &\leq 0,
	\end{align*}
	where we used $\sum_{k=t_1+1}^{t_2} 1/k \leq \int_{t_1}^{t_2} dx/x = \ln(t_2/t_1)$ in the first inequality and the last inequality comes from the assumption of the theorem. Now that the coefficients of $\E[G^k]$ are non-positive, we can simply ignore them and obtain,
	\begin{align*}%\label{eq: opt E3}
	\xi^{\tau_{i+1}}(\tau_{i+1} + \beta -1 )^2 \leq \xi^{\tau_i}(\tau_{i} + \beta -1 )^2  +  \frac{2L \sigma^2}{n \mu^2}(\tau_{i+1} - \tau_i)
	+ \frac{9L^2\sigma^2}{\mu^3}\sum_{t=\tau_i}^{\tau_{i+1}-1} \frac{t-\tau_i}{t+\beta}.
	\end{align*}
	Recursing relation above for $i=0,\ldots,R-1$ implies,
	\begin{align*}
	    \xi^{T}(T + \beta -1 )^2 \leq \xi^{0}( \beta -1 )^2  +  \frac{2L \sigma^2}{n \mu^2}T
	+ \frac{9L^2\sigma^2}{\mu^3}\sum_{t=0}^{T-1} \frac{t-\tau(t)}{t+\beta}.
	\end{align*}
	Dividing both sides by $(T+\beta -1 )^2$ concludes the proof.
\end{proof}

\begin{proof}[Proof of Theorem \ref{thm: logT}]
	We have,
	\begin{align*}
	\tau_j = \tau_0 + \sum_{i=0}^{j-1} H_i = a \frac{j(j+1)}{2}, \qquad j=0,\ldots,k-1.
	\end{align*}
	Hence,
	\begin{align*}
	    1 + \frac{H_0-1}{\tau_0 + \beta} &= 1 + \frac{a-1}{\beta} \leq 1 + \frac{T}{R^2\kappa^2},\\
	    1 + \frac{H_i-1}{\tau_i + \beta} &\leq 1 + \frac{a(i+1)}{\frac{ai(i+1)}{2}} \leq 3, \qquad i\geq 1.
	\end{align*}
	Thus, $9\kappa^2c \ln(1 + \frac{H_i - 1}{\tau_i + \beta}) + 2\kappa(1+\frac{c}{n}) - (\tau_i +1 + \beta)\leq 0, i=0,\ldots,R-1$ and we can use Theorem \ref{thm: general}.
	Moreover,
	\begin{align*}
	\sum_{t=0}^{T-1} \frac{t-\tau(t)}{t+\beta} &\leq \sum_{j=0}^{R-1} \sum_{i=1}^{H_j - 1} \frac{i}{\tau_j + i + \beta } \leq H_0 + \sum_{j=1}^{R-1} \sum_{i=1}^{H_j - 1} \frac{i}{\tau_j + 1 + \beta } \\
	&= a + \sum_{j=1}^{R-1} \frac{H_j (H_j - 1)}{2(\tau_j + 1 + \beta)} = a + \sum_{j=1}^{R-1} \frac{a(j+1) (a(j+1) - 1)}{a j(j+1) + 2(1 + \beta)} \\
	&\leq a + \sum_{j=1}^{R-1} \frac{a^2 (j+1)^2 }{aj(j+1)} \leq 2aR.
	%\qs{a + 2a(k-1)} = a(2k - 1). 
	\end{align*}
% 	Set $k = \lceil \ln(T) \rceil$. Then it's sufficient to have $ a = \lceil 2T/k^2 \rceil$ to satisfy $\tau_k \geq T$. Then,
    Plugging the values of $R$ and $a$ implies,
	\[ \sum_{t=0}^{T-1} \frac{t-\tau(t)}{t+\beta} \leq 2aR \leq 2(\frac{2T}{R^2}+1)R = \frac{4T}{R} + 2R \leq \frac{4T}{R} + \frac{4T}{R} = \frac{8T}{R}, \]
    where we used $R\leq \sqrt{2T}$ in the last inequality. Using the relation above together with Theorem \ref{thm: general} concludes the proof.
\end{proof}

\newpage
\section{More Numerical Experiments}
In this section we present more numerical experiments as well as discussion on how different hyper-parameters were selected.

\subsection{An Experiment With Real Data: Logistic Regression for Hospitalization Prediction}
We consider binary classification and select $l_2$-regularized logistic regression with its corresponding loss function as the objective function $F$ to be minimized, i.e.,
\begin{align*}
	F(\bx) = \frac{1}{N} \sum_{j=1}^N \left( \ln(1+\exp(\bx^\top \mathbf A_j)) - 1_{(b_j = 1)} \bx^\top \mathbf A_j \right) 
	+ \frac{\lambda}{2} \Vert \bx \Vert_2^2,
\end{align*}
where $\lambda$ is the regularization parameter, $\mathbf A_j \in \R^d$ and $b_j \in \{ 0,1 \}$, $j=1,\ldots,N$ are features (data points) and their corresponding class labels, respectively.
We used a real data set from the American College of Surgeons National Surgical Quality Improvement Program (NSQIP) to predict whether a specific patient will be re-admitted within 30 days from discharge after general surgery. This data set consists of $N = 722,101$ data points for training with $d=231$ features including (i) baseline demographic and health care status characteristics, (ii) procedure information and (iii) pre-operative, intra-operative, and post-operative variables.

We perform Local SGD with $n=20$ workers, $\lambda = 0.05$, $\beta = 1$, $T=500$ iterations and batch size of $b=1$  with four different communication strategies for $H$:
(i) one from \cite{stich2018local} with the choice of $\sqrt{T/(bn)} \approx 7$, (ii) one from \cite{haddadpour2019local} with the choice of $T^{2/3}/(nb)^{1/3} \approx 36$, (iii) a strategy with the time varying communication intervals with $H_i = a(i+1), a = 5$ and $R = 20$ communication rounds proposed in this paper, (iv) a strategy with the same number of communications however with a fixed $H=T/n = 50$, and finally, (v) selecting $H = T$ for one-shot averaging.
Each simulation has been repeated $10$ times and the average of their performance is reported in Figure \ref{fig: res1}.

\begin{figure}
	\centering
	\begin{subfigure}[b]{0.49\textwidth}
		\includegraphics[width=\textwidth]{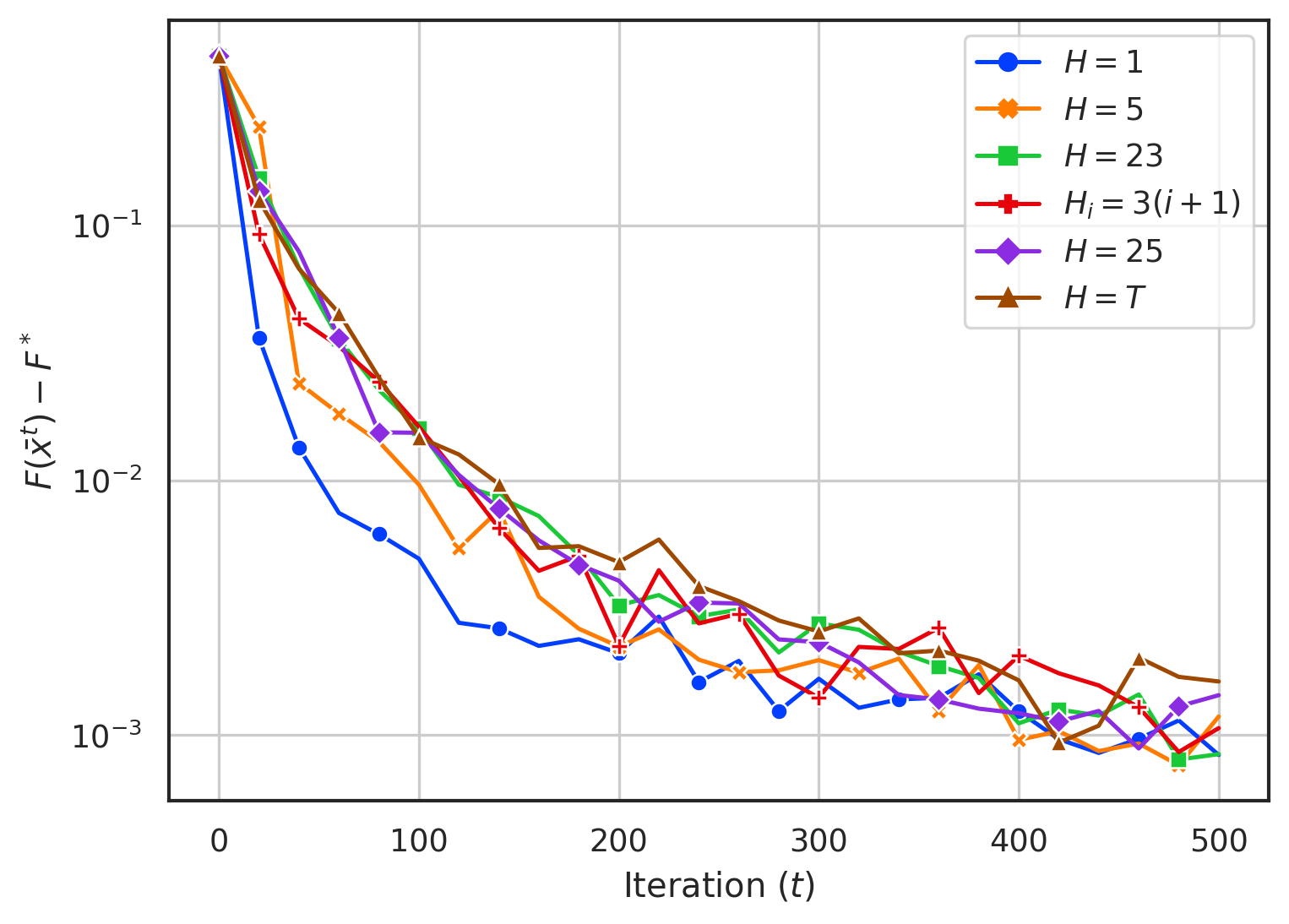}
		\caption{Error over iterations.}
	\end{subfigure}
	\begin{subfigure}[b]{0.49\textwidth}
		\includegraphics[width=\textwidth]{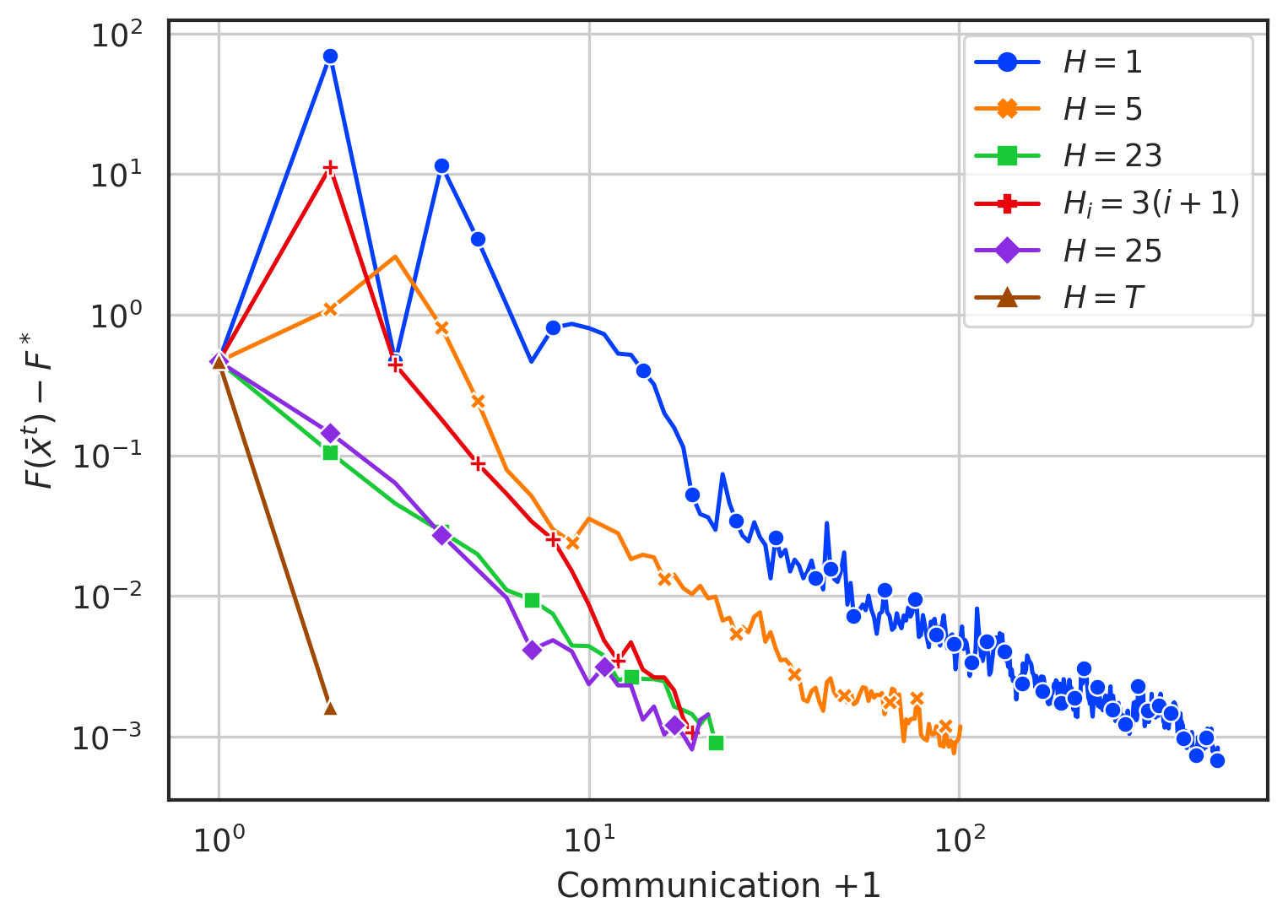}
		\caption{Error over communication rounds.}
	\end{subfigure}
	\caption{Local SGD with different communication strategies on the NSQIP data set.}
	\label{fig: res1}
\end{figure}

It can be seen that all four communication strategies have similar behavior over the number of iterations. However, the methods proposed in this paper reach the same error level with much less communication rounds, i.e., $20$ versus $143$ (\cite{haddadpour2019local}) or $28$(\cite{stich2018local}).
Surprisingly, one-shot averaging performs just as well as synchronized SGD. This could be due to the fact that our bounds analyze the worst-case scenario. Studying one-shot averaging and cases where it performs well is out of the scope of this paper and is left to future work.

We also notice that, as we get closer to the end of training, the methods with fixed-length communication intervals, have smaller improvement with each communication (see Figure \ref{fig: res1}-b). However, with the growing communication interval suggested in this paper, each communication decreases the error significantly. This further confirms that less frequent communication is needed towards the end of training.

\subsection{Logistic Regression on a9a Data Set}
Here we repeat the experiment above on the a9a data set from LIBSVM \cite{CC01a}. This data set consists of $N=32561$ data points for training with $d = 123$ features.
We use same parameters as we did for (NSQIP) data set, except this time we repeat each training $50$ times due to smaller size of the data set. The results are presented in Figure \ref{fig: res3}. Here, we observe a similar performance of different communication strategies.

\begin{figure}
	\centering
	\begin{subfigure}[b]{0.49\textwidth}
		\includegraphics[width=\textwidth]{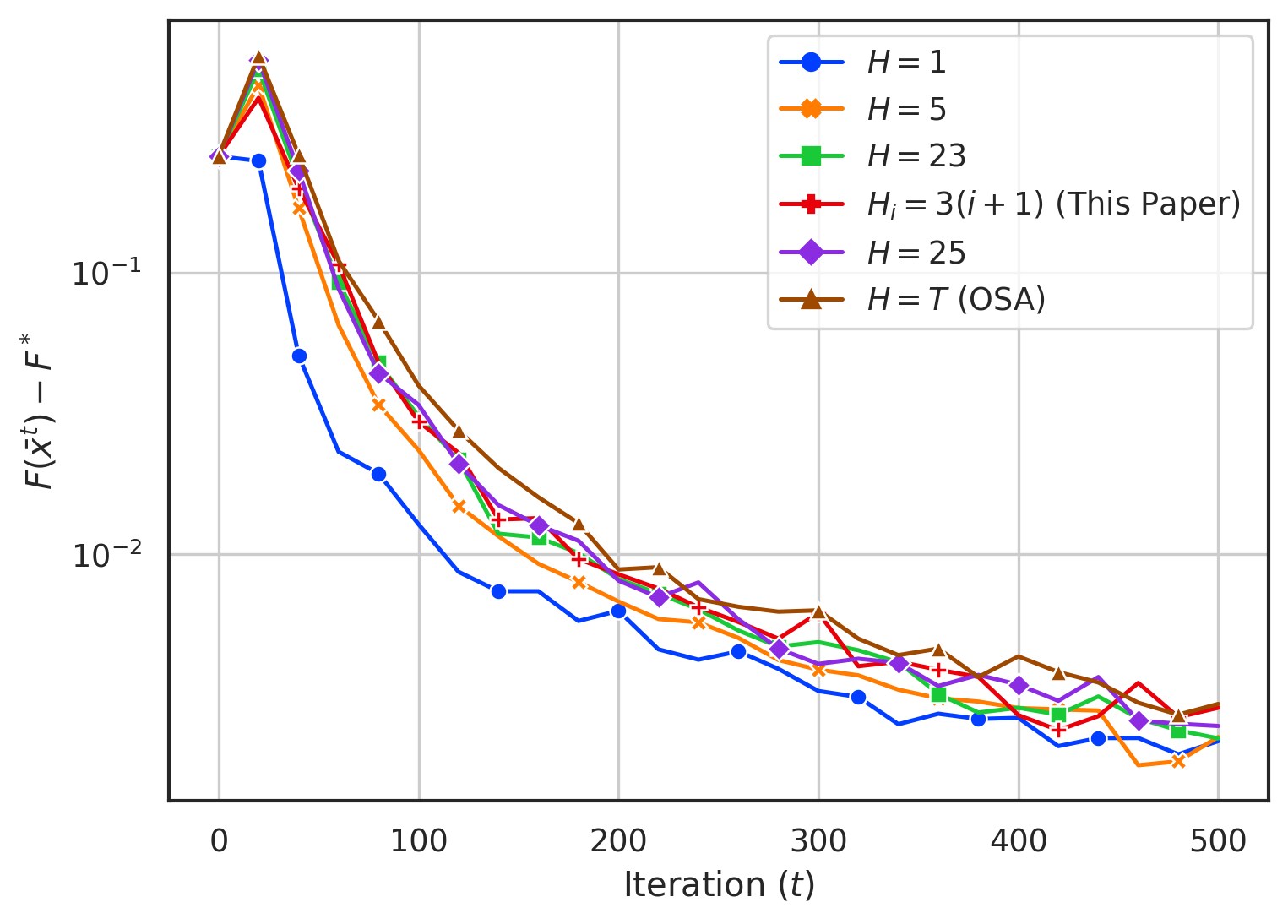}
		\caption{Error over iterations.}
	\end{subfigure}
	\begin{subfigure}[b]{0.49\textwidth}
		\includegraphics[width=\textwidth]{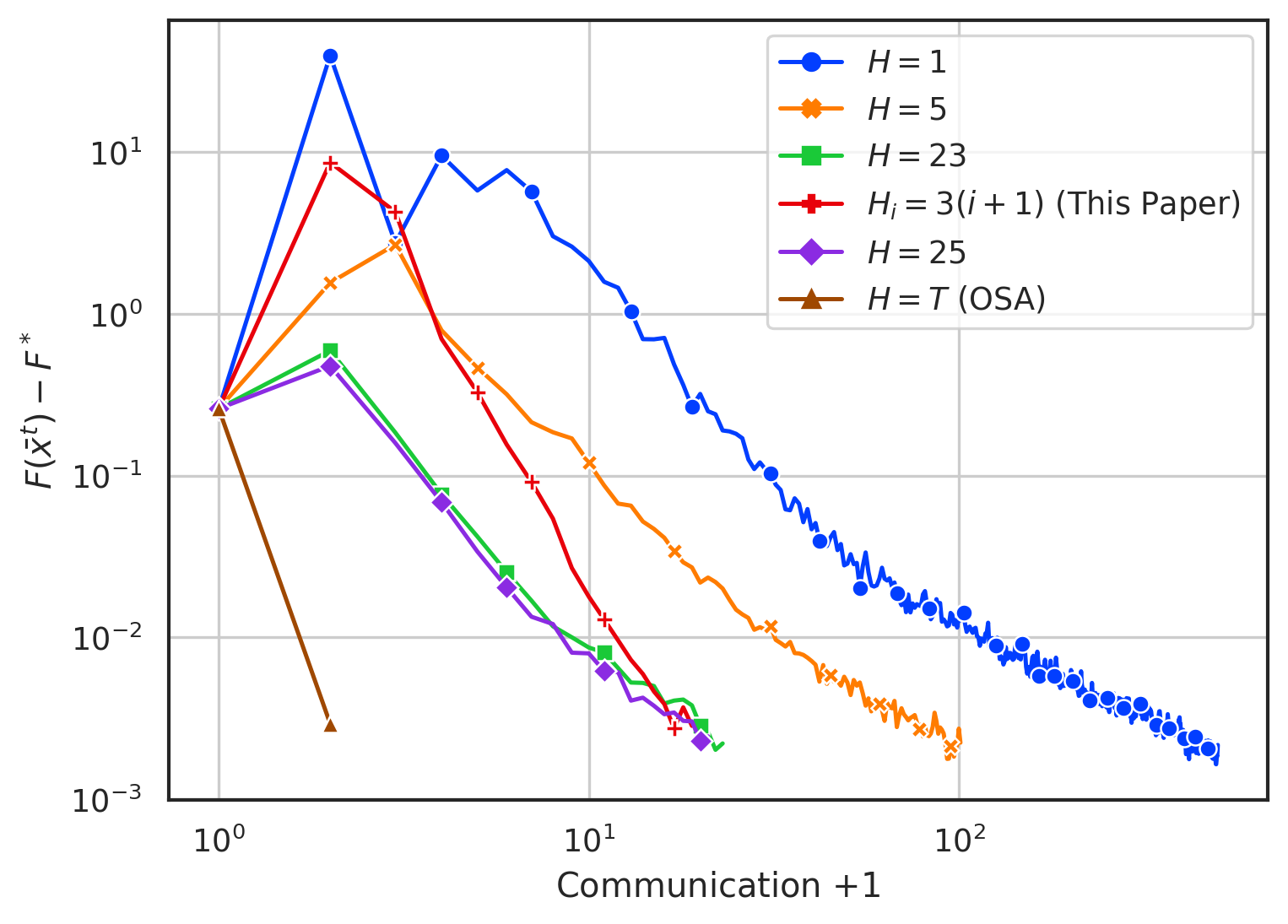}
		\caption{Error over communication rounds.}
	\end{subfigure}
	\caption{Local SGD with different communication strategies on the a9a data set.}
	\label{fig: res3}
\end{figure}

\subsection{Discussion}
We note that other communication strategies proposed in related works have often suggested their own step-size sequence or sometimes a fixed step-size. Designing an experiment to have a completely fair comparison of different methods while capturing all possible applications is not easy, if possible at all. Therefore, to make our comparison more fair, we used the same step-size sequence with $\beta = 1$ for all methods in each experiment.
Moreover, the central goal of the numerical experiments here, is to demonstrate the effectiveness of our suggested communication strategy, using minimal number of communication rounds.

\end{document}